\documentclass[a4paper,12pt,thmsa]{amsart}
\usepackage{amsmath, latexsym, amsfonts, amssymb, amsthm, amscd}
\usepackage{color}
\usepackage{amsmath,amstext,amssymb,amsopn,amsthm,mathrsfs}
\usepackage{ulem}
\DeclareMathOperator{\E}{\mathbb{E}}
\usepackage{color}
\usepackage{amsmath}
\usepackage{amssymb}
\usepackage{latexsym}
\usepackage{amsthm}
\usepackage{soul}

\newtheorem{theorem}{Theorem}
\newtheorem{corollary}{Corollary}
\newtheorem{proposition}{Proposition}
\newtheorem{lemma}{Lemma}
\newtheorem{rem}{Remark}

\newtheorem{defi}{Definition}
\newtheorem{exle}{Example}

\newcommand{\Jac}{{\rm Jac}}

\newcommand{\p}{\Bbb{P}}
\newcommand{\px}{\Bbb{P}_x}
\newcommand{\e}{\Bbb{E}}
\newcommand{\D}{\Bbb{D}}
\newcommand{\ex}{\Bbb{E}_x}

\newcommand{\R}{\Bbb{R}}
\newcommand{\A}{{\bf A}}
\newcommand{\Do}{{\bf D}}

\title[Space and time inversions of stochastic processes and Kelvin transform]
{Space and time inversions of stochastic processes and Kelvin transform}

\author{L.~Alili \and L.~Chaumont \and P.~Graczyk \and T. \.Zak}
\address{L.~Alili -- Department of
Statistics, The University of Warwick, CV4 7AL, Coventry, UK.}
\email{L.Alili@warwick.ac.uk}
\address{L. Chaumont -- LAREMA UMR CNRS 6093, Universit\'e d'Angers, 2, Bd Lavoisier \\
Angers Cedex 01, 49045, France}
\email{loic.chaumont@univ-angers.fr}
\address{P.~Graczyk  -- LAREMA UMR CNRS 6093, Universit\'e d'Angers, 2, Bd Lavoisier \\
Angers Cedex 01, 49045, France}
\email{piotr.graczyk@univ-angers.fr}
\address{T. \.Zak  --  Faculty of Pure and Applied Mathematics, Wroc{\l}aw University of
Science and Technology, Wybrze\.ze Wyspia\'nskiego 27, 50-370 Wroc{\l}aw, Poland.}
\email{tomasz.zak@pwr.edu.pl}


\date{\today}

\begin{document}


\begin{abstract}
Let $X$ be a standard Markov process.
	 We prove that a space inversion property of $X$ implies the existence of a Kelvin transform of  $X$-harmonic, excessive  and  operator-harmonic functions
	 and that the inversion property is inherited by Doob $h$-transforms.
	  We determine new  classes of processes having space inversion properties amongst transient processes  {satisfying the}  time inversion property. {For these processes,  some explicit inversions, which are often not the spherical ones, and excessive functions are given explicitly.}  We treat in details the examples of free scaled power Bessel processes, non-colliding Bessel particles, Wishart processes, Gaussian Ensemble and Dyson Brownian Motion.
\end{abstract}

\maketitle

\section{Introduction}

The following space inversion  property of  a   Brownian Motion $(B_t, t\geq 0)$ in $\R^n$
is   well known (\cite{ry}, \cite{yo}).
Let
 $I_{sph}$  be the spherical inversion $I_{sph}(x)=x/\|x\|^2$ on $\R^n\setminus\{0\}$ and $h(x)=\|x\|^{2-n}$,  $n\ge 1$.
Then  \[(I_{sph}(B_{\gamma_t}), t\geq 0)\stackrel{(d)}{=}(B^h_t,\,\, t\geq 0),\]
where $\stackrel{(d)}{=}$ stands for equality in distribution, {$B^h$ is the Doob $h$-transform of $B$ with the  function $h$ and } the time change $\gamma_t$ is  the inverse  of the {additive functional}
$A(t)=\int_0^t \|X_s\|^{-4}\, ds $. {In case  $n=1$, $B$ is a reducible process. Thus, the state space can be reduced to either  the positive or negative {half-line} and $B$   killed when it hits zero, usually denoted by $B^{0}$, is used instead of $B$.}

In \cite{bz}, such an inversion property was shown for
isotropic   ({also called} "rotationally invariant" or "symmetric") $\alpha$-stable processes on $\R^n$,
$0<\alpha \le 2$, also with $I_{sph}(x)$ and with the {excessive} function  $h(x)=\|x\|^{\alpha-n}$. The time change $\gamma_t$ is then  the inverse function of
$A(t)=\int_0^t \|X_s\|^{-2\alpha}\, ds$. In the pointwise recurrent case $\alpha > n=1$ one must consider the process $X_t^0$ killed at $0$.
In the recent papers \cite{agz, acgz, ky}, inversions  involving dual processes were studied for  diffusions on $\R$ and for  self-similar Markov processes  on $\R^n$, $n\geq 1$.\\

The main motivation and objective  of this paper are  to find new  classes of Markov processes  having space inversion properties and to study the existence of a related Kelvin transform of $X$-harmonic functions.\\


 In this work, $\left((X_t, t\geq 0); (\mathbb{P}_x)_{x\in E}\right)$ {, $X$ for short,}  is a standard Markov process with a state space $E$, where $E$ is the one point Alexandroff  compactification of an unbounded locally compact subset of $\mathbb{R}^n$.
Let $I: E\rightarrow E$ be a smooth involution and let $f$ be $X$-harmonic. One cannot expect that the function
$f\circ I $ is again $X$-harmonic. However,
 in the case of the  Brownian Motion,
it is well known, { see for instance \cite{ax}}, that
if  $f$ is a twice differentiable function on
$\R^n\setminus\{0\}$ and  $\Delta f=0$
then  $ \Delta (\|x\|^{2-n}f(I_{sph}(x)))=0$.
The map
\begin{equation*}
\label{KelvinClassic}
f\mapsto Kf(x)=\|x\|^{2-n}f(I_{sph}(x))
\end{equation*}
 is the classical Kelvin transform of a harmonic function $f$ on  $\R^n\setminus\{0\}$; this was obtained by W. Thomson (Lord Kelvin) in \cite{Thomson}.

 In the isotropic stable case,
M. Riesz noticed (\cite{riesz})	 that if
 $ K_\alpha f(x)=\|x\|^{\alpha-n}f(I_{sph}(x))$, and  $U_\alpha(\mu)$ is
  the Riesz potential of a measure $\mu$ then
 $ K_\alpha(U_\alpha(\mu))$  is $\alpha$-harmonic. This observation was extended
 in  \cite{b,btb,bz} by proving that $K_\alpha$ transforms  $\alpha$-harmonic
 functions into $\alpha$-harmonic
 functions.
Analogous results were proven for Dunkl {Laplacian} in \cite{KaYa}, see Section
\ref{KelvinIntro} for more details in the stable and Dunkl cases.

In harmonic analysis, the interest in Kelvin transform comes from the fact that
 	it reduces potential-theoretic problems relating to the point at infinity
 	for unbounded domains to those relating to the  point $0$ for bounded domains, see for instance the examples in \cite{ax} where this is applied to solving the Dirichlet problem for the exterior of the unit ball and to obtain a reflection principle for harmonic functions.

Thus, a natural question is  whether for other processes $X$,  involutions $I$ and $X$-harmonic functions $f$
one may "improve" the function $f\circ I$ by multiplying it by an $X$-harmonic function $k$ (the same for all functions $f$),
such  that the product
$$
{\mathcal K}f (x) := k(x)\, f(I(x))
$$
is $X$-harmonic. The transform ${\mathcal K}f$ will be then called  Kelvin transform of $X$-harmonic functions.

 An important result of  our paper  states that a Kelvin transform of $X$-harmonic functions
exists for any process satisfying a space inversion property.  Thus a Kelvin transform of $X$-harmonic functions exists for a much larger class of processes
	than isotropic $\alpha$-stable processes, $\alpha\in(0,2], $ and  Dunkl processes.
	Moreover, we prove that the Kelvin transform also preserves excessiveness.
	
	Throughout this paper, $X$-harmonic functions are considered, except for Section \ref{operator}, where Kelvin transform's existence is proven for operator-harmonic functions, that is for functions harmonic with respect to the
	extended generator of $X$ and the Dynkin operator of $X$.
	
	Many other important facts for processes with inversion property are proved, for instance, that the  inversion property is preserved by the Doob transform  and by  bijections. In particular, if a process $X$ has the inversion property, then so have the processes $X^h$ and $I(X)$,  {where $h$ and $I$ are as in Definition \ref{def-IP}} of Section \ref{Inversion-def} below.

New  classes of processes having space inversion properties
are determined.  We show that this is true for
transient processes with absolutely continuous semigroups that can be inverted in time.
Recall that a homogeneous Markov process {$X$} is said to have the time inversion  property (t.i.p. for short) of degree $\alpha>0$, if the process $((t^{\alpha}X_{1/t}, t\geq 0), (\mathbb{P}_x)_{x\in E})$ is homogeneous Markov.
The  processes with  t.i.p.  were intensely studied
by Gallardo and Yor \cite{Gallardo-Yor-2005} and Lawi \cite{Lawi-2008}.
 For transient processes  with  t.i.p.  we  construct appropriate space inversions and Kelvin transforms.
 A remarkable feature of this  study   is that it gives as a by-product the construction of  new  excessive functions
for processes with t.i.p.

{In Section \ref{appli} we present applications of our results to some classes of stochastic processes.
Historically{, the} first examples of processes  {satisfying the} inversion property are
Brownian Motion and stable processes. Our paper shows that  there are a lot of different examples. Dunkl processes (see Section \ref{Dunkl}) as well as  other regular processes  with  t.i.p., e.g. Wishart processes, {and
 all} $1-$dimensional diffusions
 have the  inversion property(see \cite{agz}).}
 Note {also} that we do not restrict our considerations to self-similar processes, see Section \ref{notSS}. In Section \ref{hyperb},   inversion properties for the hyperbolic Bessel process and  the hyperbolic Brownian motion(see e.g. {\cite{BGS},  \cite{pyc}, \cite{revista}} and the references therein) are discussed.\\

 {
 Here we work  with the setting  commonly used
in modern stochastic potential theory, which is provided
by the classical  {text}books
(\cite{bg,Dynkin}) and used in the recent monograph
\cite{bLN}.
In particular, we use their definitions of harmonic (and superharmonic)  functions and Doob $h$-transforms, which are more widely known.
It would be interesting to extend the results of our paper
to the setting  introduced and   used in \cite{ DellMey, meyerRef}, and, more recently, in \cite{BB1,BB2}.

\section{{Inversion property and  Kelvin transform of $X$-harmonic functions}}

\subsection{{State space for a process with inversion property}}\label{State}

{M. Yor considered in \cite{yo} the Brownian motion on $\mathbb{R}^n\cup \{\infty\}$, where $\infty$ is a point at infinity {and} $n\geq 3$.  He  was motivated by the work of  L. Schwartz \cite{schwartz}  who showed that the $n$-dimensional Brownian motion $(B_t, t\geq 0)$ on $\mathbb{R}^n\cup \{\infty\}$ is a semimartingale until time $t=+\infty$. Furthermore, the Brownian motion indexed by $[0, \infty]$ looks like a bridge between
	{the initial state $B_0$ and the  $\infty$ state}. Observe now that we can write $\mathbb{R}^n\cup\{ \infty\}=\{ \mathbb{R}^n\backslash \{0\}\} \cup \{ 0, \infty\}$. {Then}  ${S=\{\mathbb{R}^n\backslash\{0\}\} \cup \{ 0\}} $ is a locally compact space,
	{ where $0$ is an isolated  cemetery point.}  This makes sense from the point of view of involutions because we can {extend the spherical inversion on $\mathbb{R}^n\backslash \{0\}$, by  setting {$I_{sph}(0)=\infty$ and $I_{sph}(\infty)=0$}, to define an involution of $\mathbb{R}^n\cup \{\infty\}$}. }

{Following this basic case, we are now ready to fix the mathematical setting of this paper. Let $E$ be  the Alexandroff one point compactification of an {unbounded} locally compact space  $S\subset \mathbb{R}^n$. Without loss of generality, we assume that $0\in S$.
$E$ is endowed with its topological Borel $\sigma$-field.}

{We assume that $X$  is a standard process, we refer to Section
I.9 and Chapter V of \cite{bg} for an account on such processes.  That is $X$ is a strong Markov process with state space $E$.
The process $X$ is defined on some
complete filtered probability space $(\Omega,\mathcal{F},(\mathcal{F}_t)_{t\ge0},({\mathbb P}_x)_{x\in E})$, where $\mathbb{ P}_x(X_0=x)=1$,
for all $x\in  E$. The paths of $X$ are assumed to be right continuous on $[0,\infty)$, with  left limits, and are quasi-left continuous on $[0
,\zeta)$, where  $\zeta=\inf\{ s>0: X_s\notin \mathring{S}\backslash \{0\}\}$
 is the lifetime of $X$,  $\mathring{S}$  being the interior of $S$.
{Thus $X$ is absorbed at $\partial S \cup \{0, \infty\}$} and it is sent to  $0$ whenever $X$ leaves $\mathring{S}\backslash \{0\}$ through $\partial S\cup\{0\}$, and to $\infty$ otherwise. We furthermore assume that $X$ is irreducible, on $E$,  in the sense that, starting from anywhere in $\mathring{S}\backslash \{0\}$, the process  can reach with positive probability any nonempty open subset of $E$.
  This is a multidimensional generalization of  the situation
 considered in \cite{agz},  where we constructed  the dual of  a one dimensional regular  diffusion living on a  compact interval $[l,r]$ and killed upon exiting the interval. }

 Occasionally (Lemma \ref{New Lemma},  Corollary \ref{hHI}, Proposition \ref{IPforOther}, Section \ref{tip}), we will additionally assume
 that   the semigroup $p_t(x,dy)$ is absolutely continuous with respect to
 the Lebesgue measure on $E$ and write $p_t(x,dy)=p_t(x,y)dy$. Then we will briefly say that $X$ is  absolutely continuous.
\subsection{Excessive and invariant functions and Doob $h$-transform}
 {In this paper, an important role is played by Doob $h$-transform, which is defined for an  excessive function $h$.
 	 	Recall that a Borel function $h$ on $E$
 	is called {\it excessive} if
 	 $\E_x h(X_t)\le h(x)$ for all $x$ and $t$  and
{$\lim_{t\to 0+}\E_x h(X_t)=h(x) $ for all $x$.} An excessive function is said to be
 {\it invariant} if $\E_x h(X_t)= h(x)$ for all $x$ and $t$. Let $D\subset E$ be an open set. A Borel function $h$ on $E$
 	is called {\it excessive (invariant) on $D$} if it is excessive (invariant) for the process $X$ killed when it exits $D$.

  Let 
  $h$ be an excessive function 
  and set $E_h=\{ x: 0<h(x)<\infty\}$. Following   \cite{CW}, we can define the Doob $h$-transform $(X^h_t)$ of $(X_t)$  as the Markov {or sub-Markovian} process with transition semigroup prescribed by
\begin{equation*}\label{h-transformed-semi-group}
P_t^h(x,dy)=
\begin{cases}
\displaystyle{\frac{h(y)}{h(x)}}Q_t^h(x, dy) &  \hbox{if} \quad x \in E_h; \\
0 & \hbox{if} \quad x\in E\setminus E_h,
\end{cases}
\end{equation*}
where $Q_t^h(x,dy)$ is the semigroup of $X$ killed upon exiting $E_h$. Observe that if $h$ neither vanishes nor takes the value $+\infty$ inside $E$ then this killed process  is $X$ itself.

{
 Motivated by applications to Martin boundaries, {the} Doob $h$-transform is
considered by  Meyer \cite{meyerRef} and
Dellacherie-Meyer \cite{DellMey}. {Their setting } includes additional regularity properties
of the $h$-processes $X^h$.
However, for our needs, we { use the setting}
 {of}   \cite{bg, bLN, Dynkin} {since this} is more widely known.
}

\subsection{{Definition of Inversion Property (IP)}}\label{Inversion-def}   In this section,  $\left((X_t, t\geq 0); (\mathbb{P}_x)_{x\in E}\right)$, or $X$ for short, is a standard Markov process with values in a state space $E$
defined {as} in Section \ref{State}. We settle the following definition of the inversion property.
\begin{defi}\label{def-IP}
We say that $X$ has the  Inversion Property, for short IP,
if there exists an involution $I\not={\rm Id}$  of $E$ and a nonnegative $X$-{excessive} function $h$  on $E$, with {$0<h<+\infty$} in the interior of $E$,
such that the processes $I(X)$ and $X^h$ have the same law,  up to a change of time $\gamma_t$, i.e., {under $\mathbb{P}_x$, $x\in E$, we have}
\begin{equation}\label{Ih}
(I(X_{\gamma_{t}}), t\geq 0)\stackrel{(d)}{=}(X^h_t, t\geq 0),
\end{equation}
{with $X_0=x$ and $X^h_0=I(x)$,} where
$\gamma_t$ is the inverse of the additive functional $A_t = \int_0^t v^{-1}(X_s)\, ds$ with $v$ being a positive continuous function and $X^h$ is the Doob h-transform of $X$ {(killed when it exits the interior of  $E$)}. We call $(I, h, v)$ the characteristics of the {IP}. When the functions $I$ and $h$ are continuous on $\mathring{E}$,  we say that $X$ has IP with
 continuous characteristics.
\end{defi}
 We propose   the terminology    "Inversion Property" to stress the fact
	that the involuted ({\it "inversed"})  process $I(X)$  is expressed by $X$ itself, up to a Doob $h$-transform  and a time change.
	Another important point is that the IP implies that the dual process $X^h$ is obtained  by a  path transformation $I(X)$ of $X$,   up to a time change. 	
	For stochastic aspects of IP, see Definition \ref{sip} and the last part of Section \ref{basic}.

Inversion properties of stochastic processes were studied in many papers.  The IP was studied for Brownian motions  in  dimension $n\ge 3$  and for the
spherical inversion
in \cite{yo}.
  The  IP {with}   the
spherical inversion {for} isotropic   stable processes
in $\R^n$ was proved  in \cite{bz}.
The continuous case in dimension 1 was studied in \cite{agz}.  The spherical inversions of self-similar Markov processes under a reversibility condition have been studied in \cite{acgz},
 and, in the particular case of 1-dimensional stable processes in \cite{ky}.

 As pointed out above,   the  involution  involved
  in all  known  multidimensional  inversion properties (or its variants with a dual process, see
\cite{ acgz}), is spherical. On the other hand, in  the continuous one-dimensional
case, see \cite{agz},  non-spherical involutions systematically appear. In Sections \ref{tip} and \ref{appli}    of this paper we show that  many important multidimensional processes satisfy an IP with a  non-spherical involution.

\subsection{{Harmonic and superharmonic functions and their relation with excessiveness.}}
 We first recall the definitions of $X$-harmonic, regular $X$-harmonic
	and $X$-superharmonic functions on an open set $D\subset E$. For short, we will say "(super)harmonic on $D$" instead of
	"$X$-(super)harmonic on $D$", and "(super)harmonic" instead of
	"$X$-(super)harmonic on $E$".  \\
	A {Borel} function
	$f$ is  {\it harmonic  on $D$} if, for any open bounded set $B\subset\bar B\subset D,$ we have
	$$ \E_x\left(f(X_{\tau_B}), \tau_B<\infty\right)=f(x),$$
	and is {\it superharmonic}  on $D$ if
	$$ \E_x(f(X_{\tau_B}), \tau_B<\infty)\le f(x),$$
for all $x\in B$, where $\tau_B$ is the first exit time from $B$, i.e., $\tau_B=\inf\{s>0; X_s \notin B\}$. 	
	A   {Borel}  function $f$ is  {\it regular harmonic  on $D$} if
	$\E_x(f(X_{\tau_D}), \tau_D<\infty)=f(x)$. By the strong Markov property, {regular harmonicity  on $D$} implies
	{harmonicity  on $D$}.
	{In fine potential theory \cite{DellMey, meyerRef}, nearly-Borel measurable functions are also considered. For our needs and applications, we consider Borel functions, as in the settings of} \cite{bg, bLN, Dynkin}.
 Let us point out the following relations between superharmonic
and  excessive functions for standard Markov processes.

\begin{proposition}\label{harm-excess}
 Suppose that {$X$} is a standard Markov process  and let $f: E\rightarrow [0, \infty]$ be a non-negative function. Let $D\subset E$ be an open set.
\begin{itemize}
\item[(i)] If $f$ is  excessive on $D$ then $f$ is  superharmonic  on $D$.  
\item[(ii)]  If $f$ is  superharmonic  on $D$  
and $\liminf_{t\to 0+} \E_x f(X_t)\ge f(x)$,  for all $x\in D$ , then $f$ is excessive  on $D$. 
\item[(iii)] Suppose that  $f$ is a continuous function on $E$. Then $f$  is  superharmonic on $D$ 
if and only if $f$ is  excessive on $D$.
\end{itemize}
\end{proposition}}
\begin{proof}
	Without loss of generality we suppose $D=\mathring{E}$, otherwise we consider the process $X$ killed when exiting $D$.
	
	Part {(i)} is from Proposition  \cite[II(2.8)]{bg} 	of the book by Blumenthal and Getoor.
		Part  {(ii)}  is from Corollary \cite[II(5.3)]{bg}, see also Dynkin's book \cite[Theorem 12.4]{Dynkin}.\\
 In order to prove  Part  {(iii)},  {we use the right-continuity of $X_t$ when $t\to 0+$,
 the continuity of $f$ and the Fatou Lemma to see that}  the condition  from {(ii)} is fulfilled and $f$ is excessive.
 	\end{proof}

 {
 \begin{rem}
Proposition \ref{harm-excess}(iii) is essentially a particular case of  \cite[Theorem 11]{meyerRef}. Actually, the fact that a nearly-Borel mesurable superharmonic function is excessive if and only if it is finely continuous is a direct application of the theory of strongly supermedian functions developed {in}  \cite{feyel1, feyel2}. {The papers} \cite{BB1,BB2} {are more recent references on the topic}.
 \end{rem}
 }

\subsection{ Kelvin transform: definition and dual Kelvin transform}\label{KelvinIntro}	
 We shall define the  Kelvin transform for $X$-harmonic and $X$-superharmonic  functions.
 {In the Kelvin transform, only functions  on open subsets $D \subset E$ are considered. For convenience, we suppose them to be equal  to 0 on $\partial E$  (otherwise all the integrals in this section should be written
 on $\mathring{E}$, cf. \cite{bz}).}
 \begin{defi}
 	Let $I:E\rightarrow E$ be an involution.
 	We say that there exists a Kelvin transform {${\mathcal K}$} on
 	the space of $X$-harmonic functions if there
 	exists a Borel function $k\ge 0$, {on $E$,  with $k|_{\partial E}=0$}, such that  the function
 	$	x \mapsto  {{\mathcal K}f(x)=} k(x)\, f(I(x))$
 	is $X$-harmonic on $I(D)$, whenever $f$ is  $X$-harmonic on an open set $D\subset E$.
 \end{defi}


 A useful tool in the study of the Kelvin transform is provided by the dual Kelvin transform ${\mathcal K}^*$ acting on positive measures $\mu$  on $E$ and  defined formally  by
 	\begin{equation}\label{DualKelvin}
 	\int f\, d( {\mathcal K}^*\mu)= \int {\mathcal K}f d\mu
 	\end{equation}
 	for all positive Borel functions $f$ on $E$, { with $f|_{\partial E}=0$} and   ${\mathcal K}f:=k\, f\circ I$,
 	 cf. \cite{riesz, bz}.
 	Looking at the right-hand side of \eqref{DualKelvin} we see that
 	it is equal to $\displaystyle \int f(I(y))\, k(y)d\mu(y)$.
 	Consequently,
 	${\mathcal K}^*\mu=(k\mu)\circ I^{-1}= (k\mu)\circ I$, i.e.
 ${\mathcal K}^*\mu$ is simply the image (transport) of the mesure $k\,d\mu$
 by the involution $I$. This shows that   ${\mathcal K}^*\mu$ exists and is a positive measure on $I(F)$ for any  positive measure $\mu$  supported on $F\subset E$.
\\

Former results on Kelvin transform  only concern the Brownian Motion (see e.g. \cite{ax}), the isotropic $\alpha$-stable processes and the Dunkl Laplacian
	and they always refer to the spherical involution $I_{sph}(x)=x/\|x\|^2$.

In the isotropic stable case, let $ K_\alpha(f)(x)=\|x\|^{\alpha-n}f(I_{sph}(x))$.
 Riesz noticed in 1938 (see \cite[Section 14, p.13]{riesz}) the following transformation formula for the Riesz potential $U_\alpha(\mu)$ of a measure $\mu$, in the case $\alpha<n$:
$$
K_\alpha(U_\alpha(\mu))= U_\alpha(K_\alpha^*\mu),
$$
see also \cite[formula (80), p.115]{bz}. It follows that the function $K_\alpha(U_\alpha(\mu))$ is
$\alpha$-harmonic.
The $\alpha$-harmonicity of the   Kelvin transform $ K_\alpha(f)$ for all $\alpha$-harmonic functions was
proven in \cite{b,btb}. In \cite{bz}  it was strengthened to
{ regular}  $\alpha$-harmonic functions.\\[1mm]
In the Dunkl  case, let $\Delta_k$ be the Dunkl Laplacian  on $\mathbb{R}^n$ (see e.g. \cite[ Section 4C]{acgz}). Let $Ku=h\cdot u\circ I_{sph}$, where $h(x)=\|x\|^{2-n-2\gamma}$ is the Dunkl-excessive function described in
\cite[Cor.4.7]{acgz}.  In \cite[Th.3.1]{KaYa} it was proved  that if  $\Delta_ku=0$ then
$\Delta_k(Ku)=0$.

\subsection{Kelvin transform for processes with IP}
 Now we  relate the Kelvin transform to the inversion property.
In the following result we will prove that a Kelvin transform exists for processes satisfying the IP of Definition \ref{def-IP}.
The proof is based on the ideas of the proof of \cite[Lemma 7]{bz}  in the isotropic $\alpha$-stable case.

\begin{theorem}\label{PROPKelvin}
	Let $X$ be a standard Markov process.
Suppose that $X$  has the inversion  property (\ref{Ih}) with characteristics $(I, h,v)$. Let $D\subset E_h$ be an open set.
Then the Kelvin transform ${\mathcal K}f(x)=h(x)f(I(x))$ has the following properties:
\begin{itemize}
\item[(i)] If  $ f$ is  regular harmonic on $D\subset E_h$  and $f=0$ on $D^c$ then
	${\mathcal K}f$ is	 regular harmonic  on $I(D)$.
\item[(ii)] If  $f$ is   superharmonic on $D\subset E_h$ then ${\mathcal K}f$ is superharmonic  on $I(D)$.
\end{itemize}
\end{theorem}
\begin{proof} {Recall that $E_h=\{ x\in E: 0<h(x)<\infty\}$} and  consider an open set $D\subset E_h$, and $x\in D$. Let $\omega^x_D$ be the harmonic measure for the process $X$ departing from $x$ and leaving $D$, i.e. the probability law of  $X^x_{\tau^X_D}$.
In the first step of the proof,  we show that the Inversion Property of the process $X$ implies the following formula for the dual
Kelvin transform of the harmonic measure (cf. \cite[(67)]{bz})
\begin{equation}\label{67}
{\mathcal K}^*\omega^x_D = h(x)\, \omega^{I(x)}_{I(D)},\quad\quad D\subset E_h,\ x\in D.
\end{equation}
In order to show \eqref{67}, we first notice that if $Y_t=I(X_{\gamma_t})$
then
 \begin{eqnarray*}
 \tau_D^Y=\inf\{t\ge 0:\ Y_t\not\in D \}=
\inf\{t\ge 0: X_{\gamma_t} \not\in I(D) \}
= A(\tau^X_{I(D)}),
\end{eqnarray*}
 so that, for $B\subset E_h$ and $x\in D$,
we get
$$\px(Y_{\tau_D^Y}\in B, \tau_D^Y<\infty)= \p_{I(x)}(X_{\gamma(A(\tau^X_{I(D)}))}\in I(B), \tau^X_{I(D)}<\infty)=
\omega^{I(x)}_{I(D)}(I(B)).
$$
By the Inversion Property satisfied by $X$, the last probability equals
\begin{eqnarray*}
\px(Y_{\tau_D^Y}\in B, \tau_D^Y<\infty)&=&
\px((X^h)_{\tau_D^{X^h}} \in B,\ \tau_D^{X^h}<\infty)\\
&=&\displaystyle\frac1{h(x)}\ex \left(h(X_{\tau^X_D}){\bf 1}_B(X_{\tau^X_D}),\  \tau^X_D<\infty \right)\\
&=&
\frac1{h(x)} \displaystyle\int h(y) {\bf 1}_{B}(y)\omega^x_D(dy).
\end{eqnarray*}
We conclude that
  \begin{eqnarray*}h(x) \omega^{I(x)}_{I(D)}(I(B))&=&\displaystyle\int h(y) {\bf 1}_{I(B)}(I(y))\omega^x_D(dy)\\
  &=&\int {\mathcal K}{\bf 1}_{I(B)}(y) \omega^x_D(dy)\\
  &=& \int {\bf 1}_{I(B)}(y)({\mathcal K}^*\omega^x_D)(dy)
  \end{eqnarray*}
  and \eqref{67} follows. Now let $f\ge 0$ be a Borel function and $x\in I(D)$.
We have, by definition of ${\mathcal K}^*$ and by \eqref{67},
\begin{eqnarray*}
\ex {\mathcal K}f(X_{\tau^X_{I(D)}})&=&\int  {\mathcal K}f\, d\omega_{I(D)}^x
=
\int f \, d({\mathcal K}^*\omega_{I(D)}^x)\\
&=& h(x)\int f\, d\omega_{D}^{I(x)}
=
h(x)\e_{I(x)} f(X_{\tau^X_D}).
\end{eqnarray*}
Hence, if $f$ is any Borel function such that $\e_{z} |f(X_{\tau^X_D})|<\infty$ for all $z\in D$, then
\begin{equation}\label{Lemma7}
\ex {\mathcal K}f(X_{\tau^X_{I(D)}})=h(x)\e_{I(x)} f(X_{\tau^X_D}),\quad x\in I(D).
\end{equation}
Formula  \eqref{Lemma7} implies easily  the statements (i) and (ii) of the Theorem.
For example, in order to prove (ii), we consider  $f$ superharmonic on $D$.
 For any open bounded set $B\subset\bar B\subset D$ and $x\in I(B)$, we have
 $ \E_{I(x)}f(X_{\tau^X_{B}})\le f(I(x)).$ Then \eqref{Lemma7} implies that
 $$\ex {\mathcal K}f(X_{\tau^X_{I(B)}})\le h(x) f(I(x))=  {\mathcal K}f(x),$$ so $ {\mathcal K}f$ is superharmonic on $D$.
	\end{proof}
Now we show that the Kelvin transform also preserves  excessiveness of non-negative functions.
	\begin{theorem}\label{CORexcessive}
					Let $X$ be a standard Markov process.
			Suppose that $X$  has the inversion  property (\ref{Ih}) with continuous characteristics $(I, h,v)$.
		Let $D\subset E_h$ be an open set.
	If $H\ge 0$ is
	an excessive {continuous} function on $D$
	then the function ${\mathcal K}H$
	is   excessive on the set $I(D)$.	
	\end{theorem}
	\begin{proof}
		 Without loss of generality we suppose $D=\mathring{E}$, otherwise we consider the process $X$ killed when exiting $D$ and  replace $\zeta$ by  the first exit time from $D$ of $X$.
		
		Let $H$ be excessive for $X$.  {For any $\lambda >0$} we can write
		\begin{eqnarray*}
			\varphi(\lambda)&:=&\int_0^{\infty} e^{-\lambda t} \ex\left(\frac{h(X_t)}{h(x)}\frac{H\circ I(X_t)}{H\circ I(x)}, t<\zeta\right)dt\\
			&=& \int_0^{\infty} e^{-\lambda t} \ex\left(\frac{H\circ I(X_t^h)}{H\circ I(x)}, t<\zeta^h\right)dt,
		\end{eqnarray*}
		where $\zeta$ and $\zeta^h$ are the lifetimes of processes $X$ and $X^h$, respectively.
		Using  (\ref{Ih}) and making the change of variables $\gamma_t=r$, we get
		\begin{eqnarray*}
			\varphi(\lambda)&=&\int_0^{\infty} e^{-\lambda t} \e_{I(x)}\left(\frac{H(X_{\gamma_t})}{H\circ I(x)}, t<A_{\zeta}\right)dt\\
			&=&\e_{I(x)}\left(\int_0^{{\zeta}}e^{-\lambda A_r}\frac{H(X_r)}{H\circ I(x)}dA_r\right)\\
			&=&\e_{I(x)}\left(\int_0^{\zeta^{H}}e^{-\lambda A_r^H}dA_r^H\right)\\
			&=&\int_0^{\infty} e^{-\lambda t} \mathbb{P}_{I(x)}\left( t<A^H_{\zeta^H}\right)dt.
		\end{eqnarray*}

		By the injectivity of Laplace transform, we conclude that
		$$\e_x\left(\frac{h(X_t)}{h(x)}\frac{H\circ I(X_t)}{H\circ I(x)}, t<\zeta\right)=\mathbb{P}_{{I(x)}}\left( t<A^H_{\zeta^H}\right)\leq 1\; \hbox{for a.e.} \; t\geq 0.$$
		
	{By the continuity of $I,h$ and $H$, the right-continuity of $(X_t)$ and the Fatou Lemma
	we get  the excessivity inequality
		 $\e_x \left(h(X_t) H\circ I(X_t)\right)\le {h(x) H\circ I(x)}$
		 for all $t$ and $x$.}

		 Using Fubini  theorem,  we get
$$\lambda \varphi(\lambda)=1-\e_{{I(x)}}\left( e^{-\lambda A^H_{\zeta^H} }\right)\rightarrow 1, \hbox{as}\;  \lambda\rightarrow \infty,$$
because $\mathbb{P}_x(A^H_{\zeta^H}=0)=\mathbb{P}_x(\zeta^H=0)=\mathbb{P}_x(\zeta=0)=0.$
		By the Tauberian theorem, we get that
		$$\lim_{t\rightarrow 0_{+}}\e_x\left(\frac{h(X_t)}{h(x)}\frac{H\circ I(X_t)}{H\circ I(x)}, t<\zeta\right)=1.$$
		We have proven that $h\cdot H\circ I$ is excessive.
	\end{proof}

\begin{rem}
{ Theorem \ref{CORexcessive} may be also proven
 using  Proposition \ref{harm-excess}(iii)  and
	 Theorem \ref{PROPKelvin}(ii).}
	\end{rem}	

\begin{lemma}\label{New Lemma}
	Suppose that $X$ is an absolutely continuous  standard Markov process
	(i.e. the distribution   $p_t(x,dy)$ is absolutely continuous with respect to
	the Lebesgue measure on $E$ for each $x\in E$ and  $t>0$). Let $H$ be a continuous $X$-excessive function and
	let  $\tau_t$ be the inverse of the additive functional $A_t = \int_0^t v^{-1}(X_s)\, ds$ where $v>0$ is  continuous  on $E$.
	Suppose that
	\begin{equation}\label{gapp}
	\left( X^H\right)_{t\ge 0} \stackrel{(d)}{=}
	\left(X_{\tau_t}\right)_{t\ge 0}.
	\end{equation}
	Then $H$   is constant and $\tau_t=t$, for $t>0$.
\end{lemma}	
\begin{proof}
	Suppose that  the process $X$ is transient.  Let $U(x,y)=\int_0^\infty p_t(x,y)\,dt$ be the density of the potential kernel of $X$.
	We  equate the potentials of both processes
	in \eqref{gapp} and get that 	${H(y) \over H(x) }U(x,y)=v(y)U(x,y) $    for almost all $x,y\in E. $
	Hence	${H(y) \over v(y) }=H(x) $  a.s., so
	$H=const>0 $ and $ v=1 $.
	For recurrent $X$, the proof is similar.  For any open $G\subset E$, instead of the process $X$,
	we consider the process $X$ killed when entering $G$ and its potential kernel $U^G(x,y)$.
	Recall that an irreducible recurrent process starting from
	$E\setminus G$ enters $G$ with probability 1.
	We get
	${H(y) \over v(y) }=H(x) $ a.s. on $E\setminus G$ for every $G$. We conclude that  $H$   is constant and $\tau_t=t$, for $t>0$.
\end{proof}
	
\begin{corollary}\label{hHI}
	Let $X$ be a standard absolutely continuous Markov process.
Suppose that $X$  has the inversion  property  (\ref{Ih}) with  continuous characteristics $(I, h,v)$.
Then there exists $c>0$ such that the function $h\cdot h\circ I=c$ is constant on $E$. By considering, from now on, the dilated function  $h/\sqrt{c}$ in place of $h$, we have
\begin{equation}\label{equations-h-v}
h\circ I =1/h \quad \hbox{and} \quad v\circ I=1/v.
\end{equation}
\end{corollary}
\begin{proof}
Assume that $X$ satisfies (\ref{Ih}). By Theorem \ref{CORexcessive}, the function $h\cdot h\circ I$
	is excessive, {so the  Doob transform $X_t^{h\cdot h\circ I}$ is a Markov process.
	Let us compute its $\lambda$-resolvent.}

For any function {$f\in {\mathcal C}_0(E)$}, $x\in E$
{and $\lambda>0$}  we can write
\begin{eqnarray*}
\psi(\lambda)&:=&\int_0^{\infty} e^{-\lambda t} \e_x\left(\frac{h(X_t)}{h(x)}\frac{h\circ I(X_t)}{h\circ I(x)}f(X_t), t<\zeta\right)dt\\
&=& \int_0^{\infty} e^{-\lambda t} \e_x\left(\frac{h\circ I(X_t^h)}{h\circ I(x)}f(X_t^h), t<\zeta^h\right)dt.
\end{eqnarray*}
By using (\ref{Ih}) and making the change of variables $\gamma_t=r$, we obtain
\begin{eqnarray*}
\psi(\lambda)&=&\int_0^{\infty} e^{-\lambda t} \e_{I(x)}\left(\frac{h(X_{\gamma_t})}{h\circ I(x)} f(I(X_{\gamma_t})), t<A_{\zeta}\right)dt\\
&=&\e_{I(x)}\left(\int_0^{{\zeta}}e^{-\lambda A_r}\frac{h(X_r)}{h\circ I(x)}f(I(X_r))dA_r\right)\\
&=&\e_{I(x)}\left(\int_0^{\zeta^{H}}e^{-\lambda A_r^h}f(I(X^h_r))dA_r^h\right).
\end{eqnarray*}
Let  $M_r=\int_0^r (v\circ I(X_{\gamma_r}))^{-1}dr$ and let $m_r$ be the inverse of $M_r$. Using again (\ref{Ih}) and substituting  $M_r=v$, we get
\begin{eqnarray*}
\psi(\lambda)&=&\e_{x}\left(\int_0^{A_\zeta}e^{-\lambda M_r}f(X_{\gamma_{r}})dM_r\right) \\
&=&\e_{x}\left(\int_0^{\zeta}e^{-\lambda v}f(X_{\gamma_{m_v}})dv\right).
\end{eqnarray*}


{The equality of  $\lambda$-resolvents for all $\lambda>0$ and  $f\in {\mathcal C}_0(E)$
 implies the equality in law of two Markov
processes
$$
 \left( X^{h\cdot h\circ I}\right)_{t\ge 0} \stackrel{(d)}{=} \left(X_{\gamma_{m_t}}\right)_{t\ge 0}.
$$
}

The last equality implies that $X$ has the same distribution as the Doob transform
$X^{h\cdot h\circ I}$ time changed.
By applying Lemma \ref{New Lemma}, we see that $h\cdot h\circ I=c>0$ and  $\gamma_{m_t}=t$, for $t>0$.

We easily check that the inverse of $\gamma_{m_t}$ is $M_{A_t}=\int_0^t (v(X_s) v\circ I(X_s))^{-1}ds$. So $M_{A_t}=t$, $t\geq 0$, holds if and only if  $v\circ I=1/v$. Hence, equations (\ref{equations-h-v}) are proved.
\end{proof}
\begin{rem}\label{March2017Proof}
	In Corollary  \ref{hHI},
	instead of the hypothesis of absolute continuity of the process X,
	we can  consider the weaker condition on the support of the semi-group:
\begin{equation}\label{S}
\hbox{supp}(p_t(x,dy))=E, \quad x\in\mathring{ E}, t>0.
\end{equation}
	Instead of using Lemma \ref{New Lemma}, we then reason in the following way.
	
	Denote $H_m={\mathcal K}^{m-1}(h)$ for $m\ge 1$. In particular
	$H_2= {\mathcal K}(h)=h\cdot h\circ I$ and $ H_{2k}= H_2^k$.  By Theorem \ref{CORexcessive},
	all the functions  $H_m$ are excessive. 	
	By Fatou Lemma,  a pointwise limit of a sequence of
	non-negative excessive functions is an  excessive function.
	Thus $H:=\lim_k H_{2k}$ is excessive.
	Suppose that  $ H_2= h\cdot h\circ I$ is non-constant.
	By dilation of $h$,   we can suppose that
	$\inf H_2 <1$ and $\sup H_2>1$.
	Let $U=H_2^{-1}((1,\infty))$.
	The set $U$ is non empty and open in $E$ and
	$H = \infty$ on $U$.
	Start $X$ from $x_0$ such that $H_2(x_0)<1$.
	Then $H(x_0)=0$.
	But $H$ is excessive and, by (\ref{S}) we have  $\px(X_t\in U)>0$ so that
	$$0=H(x_0)\ge \E_{x_0} H(X_t)\ge \E_{x_0} (H(X_t), X_t\in U)=\infty,$$
	which is a contradiction. Thus $h\cdot h\circ I=c>0$.
\end{rem}

We  point out now the following bijective property of the Kelvin transform.
\begin{proposition}\label{PROPKelvinBIJ}
	Suppose that $X$  has the inversion  property (\ref{Ih}) with continuous characteristics $(I, h,v)$. Let $\mathcal{K}$ be the Kelvin transform.
	Then
\begin{itemize}
\item[(i)] 
$\mathcal{K}$ is an involution operator on the space of $X$-harmonic ($X$-superharmonic) functions
i.e. $\mathcal{K}\circ \mathcal{K}=Id$.
\item[(ii)] Let $D\subset E$ be an open set.
$\mathcal{K}$ is a one-to-one correspondence between the set of $X$-harmonic functions on $D$ and the  set of  $X$-harmonic  functions on $I(D)$.
\end{itemize}
\end{proposition}
\begin{proof}
 The first formula of \eqref{equations-h-v}  implies  by a direct computation that
 $\mathcal{K}(\mathcal{K}f)=f$. Then {\rm (ii)} is obvious.
\end{proof}
\subsection{{Invariance of IP by a  bijection and by a Doob transform. Stochastic Inversion Property} } \label{basic}
We shall now give some general properties of spatial inversions. We start with
the following proposition which is useful when proving that a process has IP.
Its proof is simple and hence is omitted.
\begin{proposition}\label{bijection}
	Suppose that $X$  has the inversion  property (\ref{Ih}) with characteristics $(I, h,v)$.
	Assume that
	$\Phi:E \mapsto F$
	is a bijection. Then the mapping $J=\Phi\circ I \circ \Phi^{-1}$ is an involution on $F$. Furthermore,
	the process $Y=\Phi(X)$  has IP with characteristics $(J, h\circ{\Phi^{-1}}, v\circ \Phi^{-1})$.
\end{proposition}
In the following result we prove that we can extend
the inversion property of a process $X$ on a state space $E$ to
an inversion property
for the
{Doob $H$-transform of $X$ killed on exiting from a smaller set $F\subset E$.}
\begin{proposition}\label{conditioning}
	Suppose that $X$  has the inversion  property (\ref{Ih}) with  continuous characteristics $(I, h,v)$.

	Let $F\subseteq E$ be such that
	$I(F)=F$ and suppose that there exists  an excessive
	{continuous}   function $H:F\rightarrow \mathbb{R}_+$ for $X$ killed when it exits $F$.  Consider $Y=X^H$,
	the Doob $H$-transform of $X$. Then the process $Y$ has the IP with
	characteristics $(I, \tilde h,v)$, with
	$\tilde h= {\mathcal K}H/H$,  where $ {\mathcal K}H=h\cdot H\circ I$ is the Kelvin transform of $H$.
\end{proposition}
\begin{proof}
	To simplify notation, set $Z=X^h$
	and denote by   $ \gamma_t^H$   the inverse of the additive functional $ A_t^H(t)= \int_0^t \frac{ds}{v(X^H_s)}$.
	Below,  using the properties of a time-changed  Doob transform in the first equality
	and the IP for $X$ in the second  equality,  we can write for all test functions $g$
	\begin{eqnarray*}
		\e_x\left(g(I(X_{\gamma_t^H}^H)), t< A^H_{\infty}\right)&=&
		\e_x\left(g(I(X_{\gamma_t})) \frac{H\circ I(I(X_{\gamma_t}))}{H\circ I(I(x))}, t<A_{\infty}\right)\\
		&=&
		\e_{I(x)}\left(g(Z_{t})\frac{H\circ I(Z_{t})}{H\circ I(I(x))},  t< A_{\infty}\right)\\
		&=&\e_{I(x)}\left( g(X_t) \frac{H\circ I(X_{t})h(X_t)}{H\circ I(I(x)) h((I(x)))}, t< A_{\infty}\right)\\
		&=&\e_{I(x)}\left( g(X_t) \frac{{\mathcal K}H(X_{t})}{{\mathcal K}H(I(x))}, t< A_{\infty}\right).\\
	\end{eqnarray*}
	 By Theorem \ref{CORexcessive}, the function ${\mathcal K}H  $ is $X$-excessive, so the Doob transform $X^{{\mathcal K}H}$  is well defined.
	Thus the processes $(I(X_{\gamma_t^H}^H))$ and $(X_t^{{\mathcal K}H})$ are equal in law.

	We have $X=Y^{1/H}$, so
	$X_t^{{\mathcal K}H}=Y_t^{{\mathcal K}H/H},$ and  the IP
	for the process $Y$  follows.
\end{proof}

 The aim of the following result is to   show that processes  $X^h$ and $I(X)$  inherit IP 	from  the process $X$ and to  determine the characteristics of the corresponding inversions.
 	\begin{proposition}\label{IPforOther}
 		Let $X$ be a standard absolutely continuous Markov process. Suppose that $X$  has the inversion  property (\ref{Ih}) with continuous characteristics $(I, h,v)$.
 	 Then the following inversion properties hold:
 \begin{itemize}
 		\item[(i)] The process $X^h$ has IP with characteristics   $(I, h^{-1}, v)$.
 		\item[(ii)]  The process $I(X)$ has IP with characteristics   $(I, h^{-1}, v^{-1})$.
 \end{itemize}
 	\end{proposition}
 	\begin{proof} {\rm (i)} Corollary \ref{hHI} and \eqref{equations-h-v} imply that $\mathcal{K}h/h=1/h$. The assertion follows from an application of Proposition \ref{conditioning}.\\
 	{\rm (ii)} Proposition \ref{bijection} implies that $I(X)$ has IP with characteristics   $(I, h\circ I, v\circ I)$. We conclude  using formulas (\ref{equations-h-v}).
 	\end{proof}
 	It is natural to interpret Proposition \ref{IPforOther}(i)  as
 	the converse of  the property  IP $(\ref{Ih})$.
 	\begin{defi}\label{sip}
 		We say that $X$ has  the {\bf stochastic inversion property}(SIP) with characteristics $(I,h,v)$
 		if  $X$ has  IP with characteristics $(I,h,v)$
 		and $X^h$  has  IP with characteristics $(I,h^{-1},v)$.
 	\end{defi}
 	This stochastic aspect of the inversion of the Brownian motion was not mentioned by M. Yor [33].  Up to a time change, the involution $I$ maps $X$ to $X^h$ and $X^h$ to $X$, in the sense of  equality of laws.
 	
 	Proposition \ref{IPforOther}(i) establishes
 	the existence of  SIP for absolutely continuous standard Markov processes with IP.
 	We conjecture that  all standard Markov processes with IP have SIP.
 	Remark \ref{March2017Proof} confirms the plausibility of this conjecture and shows SIP for
 	 processes verifying IP and the "full support" {condition (\ref{S}).

\subsection{Dual inversion property and Kelvin transform }
There are other types of inversions which involve weak duality, see for instance the books  \cite{bg} or \cite{CW} for a survey on duality.  Two  Markov processes $((X_t, t\geq 0); (\mathbb{P}_x)_{x\in E})$ and
$((\hat X_t, t\geq 0); (\hat{\mathbb{P}}_x)_{x\in E})$, with semigroups $(P_t)_{t\geq 0}$ and $(\hat{P}_t)_{t\geq 0}$, respectively, are in weak duality with respect to some $\sigma$-finite measure $m(dx)$ if, for all positive measurable functions $f$ and $g$, we have
\begin{equation}\label{EQNduality}
\int_{E}g(x)P_tf(x)m(dx)=\int_{E} f(x)\hat{P}_tg(x)m(dx).
\end{equation}
The following definition  is analogous to Definition \ref{def-IP},
but in place of $X$ on the right-hand side we put a dual process $\hat X$.

\begin{defi}\label{def-inv}
	Let {$X$} be a standard Markov process on $E$.
	We say that $X$ has the  Dual Inversion Property, for short DIP,
	if there exists an involution $I\not={\rm Id}$  of $E$ and a nonnegative $\hat X$-harmonic function $\hat h$  on $E$, {with $0<\hat{h}<+\infty$ in the interior of $E$,}
	such that the processes $I(X)$ and $\hat{X}^{\hat h}$ have the same law,  up to a change of time $\gamma_t$, i.e., {for all $x\in E$, we have}
	{\begin{equation}\label{DIh}
	((I(X_{\gamma_{t}}), t\geq 0), \mathbb{P}_x)\stackrel{(d)}{=}((\hat{X}^{\hat h}_t, t\geq 0), \mathbb{P}_{I(x)}),
	\end{equation}}
	where
	$\gamma_t$ is the inverse of the additive functional $A_t = \int_0^t v^{-1}(X_s)\, ds $ with $v$ being a positive continuous function,
	$\hat{X}$ is in weak duality with $X$ with respect to the measure $m(dx)$, where $m(dx)$ is a reference measure on $E$, and $\hat{X}^{\hat h}$
	is the Doob  $\hat{h}$-transform of $\hat{X}$ (killed when it exits $E$). We call $(I,\hat h,v,m)$ the characteristics of the DIP.
\end{defi}
\begin{rem}\label{selfdual}	
	We notice that if $X$ is self-dual then  IP and DIP are equivalent.
\end{rem}
 \begin{rem}\label{stable}
	Self-similar Markov processes having the DIP {with spherical inversions} were studied in \cite{acgz}.	Non-symmetric 1-dimensional  stable processes
	 were also investigated in \cite{ky} and they provide examples of processes that have
	the DIP, {{while} no  IP is known for them}.
\end{rem}
\begin{theorem}\label{KelvinDIP}
	Let $X$ have DIP property \eqref{DIh}. There exists the following Kelvin transform.
	 Let $f$ be a regular harmonic  (resp. superharmonic,  {continuous} excessive) function for the process $X$. Then
	$\hat{\mathcal K}f(x):=\hat h(x) f(I(x))$ is regular harmonic {(resp. superharmonic,  excessive)} for the process $\hat X$ (in the excessive case, one assumes  that $h$ and $I$ are continuous).
\end{theorem}
\begin{proof}
	The proof is similar to the proofs of Theorem \ref{PROPKelvin}
	and of Theorem \ref{CORexcessive}.
\end{proof}
$\;$\\
\begin{exle}
	 Let $X$ be a stable process  with index $\alpha \ge 1$ which is not spectrally one-sided. Let $\rho^-=\p_0(X_1<0)$, $\rho^+=1-\rho^{-}$ and set $x_+=\max(0,x)$, $x\in \mathbb{R}$. The function
	 {$H(x)=x_+^{\alpha\rho^-}$}  is $X$-invariant (see \cite{CabCh}), so also superharmonic  on $(0,\infty)$. Moreover  $H(0)=0$.
	 Theorem  \ref{PROPKelvin} applied to Corollary 2
	of \cite{acgz} implies the existence of the Kelvin transform for {$\alpha$-superharmonic} functions on $\R^+$, vanishing at $0$. Thus
	 \[\hat{\mathcal K}H(x)= \pi(-1) |x|^{\alpha\rho^+ -1}{\bf 1}_{\R^-}(x)\]
	 is {$\alpha$-superharmonic} on $\R^-$,
	as defined in  \cite{acgz}, $(\pi(-1),\pi(1))$ is the invariant measure of the first coordinate (angular part) of the Markov additive process(MAP) associated
to $X$.
	We conclude, by considering $-X$ in place of $X$, that the function
	 {$G(x)=x_+^{\alpha\rho^- -1}$} is superharmonic on $\R^+$.
	 It is known (see \cite{CabCh})  that
	 $G(x)$ is excessive on $(0,\infty)$.
	 It is interesting to see that the functions $H$ and $G$ are related by the Kelvin transform.  	
\end{exle}

\subsection{{IP for X  and Kelvin transform for  operator-harmonic functions}}\label{operator}
In analytical potential theory, the term  "harmonic function"  usually means $Lf=0$, for some
operator $L$. Then we say that $f$ is $L$-harmonic.

When harmonicity is defined by means of operators, we speak about {\it operator-harmonic functions}.
The main aim of this section is to prove that for standard Markov processes with IP the Kelvin transform
preserves, under some natural conditions,  the operator-harmonic property.

Note that for a Feller process $X$ with infinitesimal generator $A_X$ and state space $E$,  if $E$ is unbounded then there are no non-zero
$A_X$-harmonic functions which are in the domain Dom$(A_X)$ of $A_X$, i.e. if $f\in$ Dom$(A_X)\subset {\mathcal C}_0$ and  $A_Xf=0$ 	then $f$=0.
For this reason,  we will consider in this Section two extensions of the infinitesimal generator $A_X$:
the extended generator $\hat{\A}_X$ and the Dynkin characteristic operator $\Do_X$.

An operator $\hat{\A}_X$ is the {\it  extended} (resp. {\it full}) {\it generator}  of the  process $X$
with  domain Dom$(\hat{\A}_X)$  if for each $f\in $ Dom$(\hat{\A}_X)$, the process $(M_X^f(t), t\geq 0)$ defined, for each fixed $t\geq 0$, by
$$
M_X^f (t) = f(X_t) -  \int_0^t \hat{\A}_X f(X_s)ds
$$
is a local martingale (resp. martingale).
Extended and full generators are often used because of links with martingales, see the book \cite{EthierK} or the more recent paper \cite{PalmoR}.

For a standard Markov process $X$, its Dynkin characteristic operator $\Do_X$ is defined by
\begin{equation}
\label{dynkin_op}
\Do_X f(x)=\lim_ {U\searrow \{x\}}\frac{\ex f(X_{\tau_U})-f(x)}{\ex \tau_U},
\end{equation}
with $U$ being any sequence of decreasing bounded open sets such that $\cap U=\{x\}  $
(see \cite{Dynkin}, where $\Do_X$ is denoted by ${\mathcal U}$).

We stress that the extended generator and the Dynkin characteristic operator exist and characterize all standard Markov processes.

It is known that when $X$ is a diffusion,  the domains of  $\hat{\A}_X$
and of $\Do_X$ contain  ${\mathcal C}^2$ and that  the extended  generator $\hat{\A}_X$
coincides on  ${\mathcal C}^2$  with the Dynkin characteristic operator $\Do_X$
(see \cite[Prop. 3.9, p.358]{ry},  \cite[(5.18)]{PalmoR} and  \cite[5.19]{Dynkin}).
The extended operator $\hat{\A}_X$ restrained  to $\mathcal{C}^2$ is the second order elliptic  differential operator coinciding with the infinitesimal generator $A_X$ of $X$
on its domain Dom$(A_X)\subset {\mathcal C}_0\cap  {\mathcal C}^2$.

The following property of operators $\hat{\A}_X$ and $\Do_X$ is straightforward to prove.

\begin{proposition}\label{DynkinDoob1} 	Let $X$ be a standard  Markov process and let $\varphi$ be a homeomorphism from $E$ onto $E$.
	\begin{itemize}
		\item[(i)] We have  $f \in$ {\rm Dom}$(\hat{\A}_{\varphi(X)})$ if and only if $f\circ \varphi\in$ {\rm Dom}$(\hat{\A}_X)$ and
		\begin{equation*}
		{\hat{\A}_{\varphi(X)} f= [\hat{\A}_X(f\circ \varphi)]\circ \varphi^{-1}.}
		\end{equation*}
		\item[(ii)]	 We have  $f \in$ {\rm Dom}$(\Do_{\varphi(X)})$ if and only if $f\circ \varphi\in$ {\rm Dom}$(\Do_X)$ and
		\begin{equation*}
		{\Do_{\varphi(X)} f= [\Do_X(f\circ \varphi)]\circ \varphi^{-1}.}
		\end{equation*}
	\end{itemize}	
\end{proposition}

In the next proposition we present known results on the formula for the extended  generator $\hat{\A}^h$ of the Doob $h$-transformed process $X^h$.
In order to formulate them, let us recall the notion of a {\it good function in Palmowski--Rolski sense} (PR-good function for short), introduced in \cite[(1.1), p.768]{PalmoR} as follows.

Consider a Markov process $X$ having extended generator $\hat{\A}_X$ with domain Dom$(\hat{\A}_X)$.  For each strictly positive Borel function $f$
define
$$
E^f(t)=\frac{f(X(t))}{f(X(0))}\exp\left(-\int_0^t \frac{(\hat{\A}_Xf) (X(s))}{f(X(0))}\, ds\right),\ \ \ t\ge 0.
$$
If, for some function $h$, the process $E^h(t)$ is a martingale, then it is said to be an exponential martingale and in this case we call $h$ a {\it good} function.

If $\inf_x h(x)>0$ then  $E^h(t)$ is a martingale with respect to the standard filtration $(\mathcal{F}_t)$ if and only if  $M^h_X(t)$ is a martingale with respect to the same filtration (see \cite[Lemma 3.2, page 174]{EthierK}).\\
Simple sufficient conditions for a PR-good function are given  in \cite[Prop.3.2(M1)]{PalmoR}. Namely, if $h\in{\mathcal{M}}_b$(the space of bounded measurable functions) and $h^{-1}\hat\A_X h\in{\mathcal{M}}_b$ then $h$ is a  PR-good function.
If additionally we suppose that $\hat{\A}_X h=0$ then   \cite[Prop.3.2(M1)]{PalmoR} implies that the condition $h\in{\mathcal{M}}_b$ guarantees that  $h$ is a  PR-good function.  Other  sufficient conditions for
$E^h(t)$ to be a martingale
could be also deduced from \cite{Yor+2}. \\[1mm]
In the part (iii)  of Proposition \ref{DynkinDoob2} we prove a formula for $\Do_X^h$, the Dynkin characteristic operator of $X^h$.   

\begin{proposition}\label{DynkinDoob2}
Let $X$ be a standard Markov process. Suppose that $h$ is excessive for $X$.
	\begin{itemize}
		\item[(i)]
		If $X$ is a diffusion then the  extended generator $\hat{\A}^h$ of the Doob $h$-transform $X^h$ of $X$ is given, for $f\in {\mathcal{C}}^2$, by
		\begin{equation}\label{EXTgenerator_h_process}
		\hat{\A}^h(f)= h^{-1} \hat{\A}_X(hf).
		\end{equation}
		\item[(ii)] If $h$ is PR-good and  $\hat{\A}$-harmonic then	 \eqref{EXTgenerator_h_process} holds true.
		\item[(iii)] 
		The Dynkin operator $\Do_X^h$ of the Doob $h$-transform $X^h$ of $X$ is given by the formula:
		\begin{equation}\label{generator_h_process}
		\Do_X^h(f)= h^{-1} \Do_X(hf).
		\end{equation}
	\end{itemize}
\end{proposition}
\begin{proof}
	(i)  This is given in \cite[Prop. 3.9, p.357]{ry}. \\
	(ii) The statement follows from \cite[Theorem 4.2]{PalmoR}. \\	
	(iii)	Let $\lambda>0$. The $\lambda$-potential of the $h$-process $X^h$ equals
	$$ U^h_\lambda(x,dy)= \frac{h(y)}{h(x)} U^{X}_\lambda(x,dy)$$ where  $ U^{X}_\lambda$ is the $\lambda$-potential of $X$.
	Let $B$ be  the Dynkin operator of the process $X^h$. Define
	$$
	K_\lambda f= h^{-1} \Do_X(hf)-\lambda f
	$$
	To prove (iii) it is enough to show that $B-\lambda Id=K_\lambda$.
	This in turn will be proved if we show that
	$$K_\lambda   U^h_\lambda =-Id$$
	(since
	$(B-\lambda Id) U^h_\lambda=-Id$, the $\lambda$-potential operator $U_\lambda^h$ is a bijection  from   ${\mathcal C}_0$
	into the domain of  $B$  and $ B -\lambda Id$ is the unique inverse operator).	We compute, for a test function $f$,
	\begin{eqnarray*}
		K_\lambda   U^h_\lambda f&=&\frac{1}{h(x)} \Do_X[h(x) \int  \frac{h(y)}{h(x)} U_{\lambda}^X(x,y) f(y)dy]-\lambda\int  \frac{h(y)}{h(x)} U_{\lambda}^X(x,y) f(y)dy\\
		&=&\frac{1}{h(x)}( \Do_X-\lambda Id) { U_\lambda}^X (hf)\\
		&=&\frac{1}{h(x)}(-h(x)f(x))=-f(x),
	\end{eqnarray*}
	hence  $K_\lambda   U^h_\lambda =-Id$.	
\end{proof}

In the following main result of this subsection, we show
that if $X$ has the property IP, then  the Kelvin  transform preserves  the operator-harmonicity property for extended generators (under some mild additional hypothesis) and
for Dynkin characteristic operators.

\begin{theorem}\label{DynkinKelvin}
Suppose that $X$  has the inversion  property (\ref{Ih}) with characteristics $(I, h,v)$.
	Let  ${\mathcal K}H(x)= h(x)\, H(I(x))$ be the corresponding Kelvin transform.
\begin{itemize}
	\item[(i)] If $X$ is a diffusion, the characteristics of IP are continuous and  $H$ is $\hat{\A}_X$-harmonic and  twice continuously differentiable  on  an   open set  $D\subset E$,  then
	$\hat{\A}_X({\mathcal K}H) =0$ on $I(D)$.  
	\item[(ii)] If $X$ is a standard Markov process, $h$ is a PR-good function and  $H$ is $\hat{\A}_X$-harmonic   on  $D$,  then
	$\hat{\A}_X({\mathcal K}H) =0$ on $I(D)$.
	\item[(iii)]  If $X$ is a standard Markov process, and $H$ is a $\Do_X$-harmonic function on $D$ then  $\Do_X({\mathcal K}H) =0$ on $I(D)$.
\end{itemize}
\end{theorem}
\begin{proof}
	We first prove (ii) and (iii). Their proofs are identical
	and based on Propositions \ref{DynkinDoob1} and
	\ref{DynkinDoob2}, hence we present only the proof of (iii).\\
	By Proposition \ref{DynkinDoob1}(ii) we have
	\begin{equation*}\label{FORM_Kelvin}
	\Do^I(\tilde H)=\Do_X(\tilde H\circ I)\circ I^{-1}=(\Do_XH)\circ I=0.
	\end{equation*}
	Thus $\tilde H$ is $\Do^I$-harmonic on $I(D)$.
	By IP, this is equivalent  to be $\Do_X^h$-harmonic
	(the Dynkin operators  of $I(X)$ and $X^h$ differ by a positive factor
	corresponding to the time change, see \cite{Dynkin}, Th. 10.12). Consequently
	$\Do_X^h(\tilde H)=0$. We  now  use
	Proposition \ref{DynkinDoob2}(iii)
	in order to conclude that $\Do_X(h\tilde H)=0$.
	Thus $h\tilde H=h H\circ I$ is $\Do_X$-harmonic on $I(D)$ whenever $H$
	is $\Do_X$-harmonic on $D$. 	\\
	(i) 
	By 
	(iii), we have $\Do_X({\mathcal K}f)=0$. By the continuity
	of $H, I$ and $h$, the function ${\mathcal K}H$ is continuous. Theorem 5.9 of
	\cite{Dynkin} then implies that ${\mathcal K}H$ is twice  continuously differentiable and that $\Do_X({\mathcal K}H) =0$.
\end{proof}
We end this section by pointing out  relations between $X$-harmonic functions on a subset $D$ of $E$  and Dynkin $\Do_X$-harmonic functions on $D$.
\begin{proposition}\label{equiv_harmonic} {Let $X$ be a standard Markov process, $D\subset E$ and $f: D\rightarrow \mathbb{R}$.  The following assertions hold true.}
	\begin{itemize}
		\item[(i)]  If $f$ is  $X$-harmonic  then
		$\Do_Xf=0$, on $D$.
		\item[(ii)]  If  $X$  is a diffusion and  $f$  is  continuous then  $f$ is  $X$-harmonic  if and only if  it is
		$\Do_X$-harmonic,  on $D$. Moreover, this happens if and only if $f$ is $\hat{\A}_X$-harmonic on $D$.
	\end{itemize}
\end{proposition}
\begin{proof}
	Part (i) is evident by definition \eqref{dynkin_op} of $\Do_X$.
	It gives the "only if" part of the first part of (ii). If $f$ is continuous and $\Do_X$-harmonic  on $D$ then,  by Theorem 5.9 of
	\cite{Dynkin}, $f$ is twice  continuously differentiable and $\hat{\A}_X f =0$ on $D$. A strengthened version of Dynkin's formula  	\cite[(13.95)]{Dynkin} implies that if $\hat{\A}_X f =0$ on $D$ then  $f$ is  $X$-harmonic on $D$. This completes the proof of (ii).
\end{proof}
\begin{rem}
	Theorem \ref{DynkinKelvin}(iii) and Proposition  \ref{equiv_harmonic}(ii)
	give another "operator-like" proof of Theorem \ref{PROPKelvin} 
	when $X$ is a  diffusion 	and for continuous $X$-harmonic functions, see also Remark 7 in \cite{agz}.
\end{rem}
\begin{rem}
	Proposition  \ref{equiv_harmonic}(ii) suggests that there should be, under some mild  assumptions, an equivalence
	between the $\hat{\A}_X$-harmonicity and the
	property that $f(X)$ is a local martingale (martingale for full generator). 
  One implication is obvious. That is, if $f$ is $\hat{\A}_X$-harmonic then $f(X)$ is a local martingale (martingale for full generator).
\end{rem}
\begin{rem}
	It seems plausible
	that   Proposition \ref{DynkinDoob2}(ii) holds for any $X$-excessive function $h$ in place of a PR-good function.  Consequently,  when $X$ has IP, the existence of Kelvin transform would be proven   for $\hat{\A}_X$-harmonic functions.\\
	Observe that for the Dynkin characteristic operator, Proposition \ref{DynkinDoob2}(iii) has no additional hypotheses on the excessive function $h$. 	Note also that Dynkin \cite[p. 16]{Dynkin} introduces  {\it quasi-characteristic operators}, clearly related with the martingale property. We claim that under some mild regularity conditions:  $Dom(\hat{\A}_X)\subset  {\mathcal C}$
	and $\hat{\A}_X(Dom(\hat\A_X)) \subset  {\mathcal C}$, 	the extended generator $\hat{\A}_X$ coincides with the quasi-characteristic Dynkin operator, so also with
	the characteristic Dynkin operator (see \cite[p.16]{Dynkin}).
\end{rem}	

\subsection{ {Inversion property and self-similarity}} \label{notSS}
We end this Section by {a discussion} on the relations between the IP and self-similarity.
In \cite{agz} the IP of non necessarily self-similar one-dimensional diffusions is proven and
corresponding non-spherical involutions are given.
There are $h$-transforms of Brownian motion on intervals which are not self-similar Markov processes.  On the other hand IP is preserved by conditioning, see Proposition \ref{conditioning},
  but self-similarity is not.

  This shows that self-similar Feller processes are not the only ones having the inversion property with the spherical inversion and a harmonic function being a power of the modulus.


\section{Inversion of processes having the time inversion property}\label{tip}
\subsection{{Characterization and regularity  of processes with  t.i.p.}}

Now let us introduce a class of processes that can be inverted in time. Let $S$ be a  non trivial  cone of $\mathbb{R}^n$, for some $n\geq 1$, i.e. $S\neq \emptyset$, $S\neq \{0\}$ and $x\in S$ implies $\lambda x\in S$ for all $\lambda \geq 0$. We take $E$ to be the  Alexandroff one point compactification $S\cup \{\infty\}$ of $S$.
Let $\left((X_t, t\geq 0); (\mathbb{P}_x)_{x\in E}\right)$ be a homogeneous Markov process on $E$ absorbed at $\partial S \cup \{\infty \}$.
 $X$ is said to have the {\it time inversion  property} (t.i.p. for short) of degree $\alpha>0$, if the process $((t^{\alpha}X_{1/t}, t\geq 0), (\mathbb{P}_x)_{x\in E})$ is a homogeneous Markov process. Assume that the semigroup of $X$ is absolutely continuous with respect to the Lebesgue measure, and write
 \begin{equation}\label{absolute-continuity}
 p_t(x,dy)=p_t(x,y)dy, \quad   x, y\in \mathring{S}.
  \end{equation}
  The process $(t^{\alpha} X_{\frac1t}, t>0)$ is usually an inhomogenous Markov process with transition probability densities $q_{s,t}^{(x)}(z,y)$, for $s<t$ and $x, y\in S$, satisfying
\[\mathbb{E}_x\left(f(t^{\alpha} X_{\frac1t})| s^{\alpha} X_{\frac1s}=z\right)=\int f(y) q_{s,t}^x(z,y)\, dy \]
where
\begin{equation}\label{semigroup-densities-tip}
q_{s,t}^{(x)}(a,b)=t^{-n\alpha} \frac{p_{\frac1t}(x, \frac{b}{t^{\alpha}}) p_{\frac1s-\frac1t}( \frac{b}{t^{\alpha}}, \frac{a}{s^{\alpha}})}{p_{\frac1s}(x, \frac{a}{s^{\alpha}})}.
\end{equation}

{We shall now extend the setting and conditions considered by Gallardo and Yor in \cite{Gallardo-Yor-2005}}. Suppose that
\begin{equation}\label{densities}
p_t(x,y)=t^{-n\alpha/2}\phi(\frac{x}{t^{\alpha/2}}, \frac{y}{t^{\alpha/2}})\theta(\frac{y}{t^{\alpha/2}})\exp\{- \frac{\rho(x)+\rho(y)}{2t} \},
\end{equation}
where the functions $\phi: \mathring{S} \times \mathring{S}\rightarrow \mathbb{R}_+$ and $\theta, \rho : \mathring{S} \rightarrow \mathbb{R}_+$ satisfy the following properties: for $\lambda>0$ and $x, y \in \mathring{S}$
\begin{align}\label{conditions-functions}
 \left\{
  \begin{array}{l l}
    \phi(\lambda x, y)=\phi( x, \lambda y), & \\
   \rho(\lambda x)= \lambda ^{2/\alpha}\rho(x), &\\
   \theta (\lambda x)= \lambda ^{\beta} \theta(x). &
  \end{array} \right.
\end{align}
Under conditions (\ref{densities}) and (\ref{conditions-functions}), using  (\ref{semigroup-densities-tip}) we immediately conclude that $X$ has the time inversion property. We need also the following technical condition
\begin{eqnarray}\label{condition-Bessel}
&(\rho^{ {1/2}}(X_t), t\geq 0 ) \ \ \text{
is a  Bessel process of dimension}\ \ (\beta+n)\alpha\\
&\text{ or is a Doob transform of it, up to  time scaling}\ \
	t\to ct, c>0. \nonumber
\end{eqnarray}
To simplify notations let us settle the following definition
of a regular process with t.i.p.

   {
  \begin{defi}\label{regula-tip} A {\bf  regular process with t.i.p.} is a Markov process on $ S\cup \{\infty\}$ where $S$ is a non-trivial cone in $\mathbb{R}^n$ for some $n\geq 1$,  with an absolutely continuous semigroup with densities satisfying conditions \eqref{densities}--\eqref{condition-Bessel} {and}  $\rho(x)=0 $ if and only if $x=0$.
\end{defi}
}

The requirement of {\it regularity} for a  process with t.i.p. is not {very} restrictive; all the known examples of processes with t.i.p. satisfy it.  In case when $S=\mathbb{R}^n$, the authors of \cite{Gallardo-Yor-2005} and \cite{Lawi-2008}  showed that  if the above densities  are twice differentiable in the space and time  then $X$ has  time inversion property if and only if it has a semigroup with densities of the  form (\ref{densities}), or if $X$  is  a Doob $h$-transform of a process  with a semigroup with densities of the form (\ref{densities}).   It is proved in \cite{aay} that  when $\mathring{S}=\mathbb{R}$ or $(-\infty, 0)$ or $(0, +\infty)$ and the semigroup is conservative, i.e. $\int p_t(x,dy)=1$, and absolutely continuous with densities which   are twice differentiable  in time and space, then  (\ref{condition-Bessel}) is necessary for the t.i.p. to hold.
 A similar statement is proved in \cite{ay} in higher dimensions under the additional condition that $\rho$ is continuous on $S=\mathbb{R}^n$ and  $\rho(x)=0$ if and only if $x=0$.
\begin{rem}
 {Under the conservativeness condition,  it is an interesting problem to find a way} to read the dimension of the Bessel process $\rho^{1/2}(X)$, in (\ref{condition-Bessel}), from (\ref{densities}).  If we could do that then we would be able to replace condition  (\ref{condition-Bessel}) with the weaker condition that $\rho(X)$ is a strong Markov process. Indeed, it was proved in \cite{aay} that the only processes having the t.i.p. living on $(0, +\infty)$ are $\alpha$ powers of Bessel processes and their $h$-transforms. $\rho(X)$ has the time inversion property and so, if it is Markov then it is the power of a Bessel process or a process in $h$-transform with it.
\end{rem}
\subsection  {{ A natural involution and IP for   processes with  t.i.p.} }
\begin{proposition}{
The map $I$ defined  for	$x\in S\backslash \{0\}$   by  $I(x)=x\rho^{-\alpha}(x)$, and by
 $I(0)=\infty$,
   is an involution of $E$.
  Moreover, the function  $x\rightarrow x\rho^{-\nu}(x)$
is an involution on {$S\backslash \{0\}$} if and only if $\nu=\alpha$.}
\end{proposition}
\begin{proof} {It is readily checked that} $I\circ I=I$ { by using} the homogeneity property of $\rho$ from
	\eqref{conditions-functions}.
\end{proof}
 We know by \cite{Gallardo-Yor-2005, Lawi-2008} that  a regular process with t.i.p. $X$ is a self-similar Markov process, thus so is
  $I(X)$. That is why $I(x)=x\rho^{-\alpha}(x)$ is a  natural involution for such an $X$.

{ We {now} compute  the potential of the involuted process $I(X)$.
\begin{proposition}\label{TWOpotentials} Assuming that $X$ is transient for compact sets, the potential of
$I(X)$  is given by
\begin{eqnarray}\label{potentials}
U^{I(X)}(x,dy)
= V(y) \frac{h(y)}{h(x)} U^{X}(x,dy),
\end{eqnarray}
where $ h(x)=\rho(x)^{1-(\beta+n)\alpha/2}$, $V(y)=\Jac(I)(y)\rho(y)^{n\alpha -2}$
and $\Jac(I)$ is the modulus of the Jacobi determinant of $I$.
\end{proposition}
\begin{proof} Recall that $X$ is transient for compact sets if and only if its potential $U^{X}(x,y)$ is finite. The potential kernel of $I(X)$ is given by
\begin{eqnarray*}
U^{I}(x,y)&=&\int_0^{\infty} p_t(I(x), I(y))\Jac(I(y))dt.\\
\end{eqnarray*}
First we compute $p^{I(X)}_t(x,y)=p_t(I(x), I(y)) \Jac(I(y))$. According to formula (\ref{densities}) we find
$$
p^{I(X)}_t(x,y)=t^{-(n+\beta)\alpha/2}\phi(x,\frac{y}{(t\rho(x)\rho(y))^\alpha})\,\rho^{-\alpha\beta}(y) \theta(y) \exp[-\frac{\rho(x)+\rho(y)}{t\rho(x)\rho(y)}] \Jac(I(y)).
$$
Making the substitution $t\, \rho(x)\, \rho(y)=s$ we obtain easily  formula (\ref{potentials}).
\end{proof}

We are now ready {to prove}
 the main result of this section.
\begin{theorem}\label{IPtip}
Suppose that  $X$  is a transient  regular process with t.i.p.
 Then $X$ has
 the IP  with characteristics $(I,h,v)$   with
 $I(x)=x\rho^{-\alpha}(x)$,
   $h(x)=\rho(x)^{1-(\beta+n)\alpha/2}$
 and
$v(x)={(\Jac(I)(x))^{-1}\rho(x)^{2-n\alpha}}$,  where
$\Jac(I)$  is the modulus of the Jacobi determinant of $I$.

Moreover,  if $X$ is  the Doob  $h$-transform   of a regular process $Z$ having IP with characteristics $(I, h, v)$, then $X$ has the IP with characteristics  $I$ and $v$, and excessive function  ${\mathcal K_Z}(H)/H$.
\end{theorem}
\begin{proof}
First suppose that the process $X$ is regular, so its semigroup  has the form \eqref{densities}.
We use the fact that if two transient Markov processes have equal
potentials $U^X=U^Y<\infty$ then the processes $X$ and $Y$ have the same law ({compare with} \cite{hunt}, page 356 or \cite{meyer}, Theorem T8, page 205).

 {Remind that the function $h(x)=x^{2-\delta}$ is BES$(\delta)$-excessive, see e.g.
 	\cite[Cor.4.4]{acgz}. This can also be explained by the fact that}
  if $(R_t, t\geq 0)$ is a Bessel process of dimension $\delta$ then $(R^{2-\delta}_t, t\geq 0)$ is a local martingale (it is a strict local martingale when $\delta>2$), cf. \cite{eliyor}.

   Using {condition} \eqref{condition-Bessel}, we see that the function
$ h(x)=\rho(x)^{1-(\beta+n)\alpha/2}$ appearing in \eqref{potentials}
is $X$-excessive.
Thus the process $I(X)$ is a Doob $h$-transform  of the process $X$ when time-changed appropriately.

In the case where $X=Z^H$ is a Doob $H$-transform of $Z$ whose semigroup  has the form
\eqref{densities}, we use Proposition \ref{conditioning}.
\end{proof}

\begin{rem}\label{newHarm}
A remarkable  consequence of  Theorem \ref{IPtip}  is that it gives as a by-product the construction of  new  {excessive} functions which are functions of $\rho(X)$ and not of $\theta(X)$. For example, for Wishart processes,
the known harmonic functions are in terms of $\det(X)$ and not of ${\rm Tr}(X)$, see  \cite{Donati-Doumerc-Matsumoto-2004} and Subsection \ref{Wishart} below.
\end{rem}
In view of applications of Theorem \ref{IPtip},  the aim of the next result is to give a sufficient condition for $X$ to be transient for compact sets.
\begin{proposition}\label{condition-transience}  Assume that $\phi$ satisfies
\begin{itemize}
\item[(a)]
$\phi(x, y/t) \approx c_1(x, y) t^{\gamma_1(x,y)}e^{-\frac{c_2(x,y)}{t}}$ as $t\rightarrow 0$;
\item[(b)] $\phi(x, y/t) \approx c_3(x, y) t^{\gamma_2(x,y)}$ as $t\rightarrow \infty$;
\end{itemize}
where $c_1, c_2$, $c_3$ and $\gamma_1$, $\gamma_2$
 are functions of $x$ and $y$. If
\begin{itemize}
 \item[(1)]
 $\rho\geq 0$;
 \item[(2)]
$\rho(x)+\rho(y)-2c_2(x,y)>0$ for all  $x,y \in E $;
\item[(3)]
$\gamma_1(x,y)> -1+\frac{(n+\beta)\alpha}{2}>\gamma_2(x,y) $;
\end{itemize}
then $X$ is transient for compact sets.

\end{proposition}
\begin{proof}
We {easily} check that the integral for $U^X(x,y)$ converges if the hypotheses  of the  proposition are satisfied.
\end{proof}
\subsection{{Self-duality for processes with t.i.p.}}

\begin{proposition}\label{duality}
	Suppose that $\phi(x,y)=\phi(y,x)$ for $x,y \in E$. Then the process $X$ is self-dual with respect to the measure
	$$m(dx)=\theta (x)dx.$$
\end{proposition}
\begin{proof}
	Formula (\ref{densities}) implies that the kernel
	$$
\tilde  p_t(x,y):=	p_t(x,y) \theta (x)
	$$
	is symmetric, i.e. $\tilde  p_t(x,y) =\tilde  p_t(y,x)$.
	 It follows that for all $t\ge 0$
		and {bounded measurable functions $f$, $g:E\rightarrow \mathbb{R}^+$, we have}
	$$
	\int f(x) \e_x(g(X_t)) m(dx)= \int \e_x( f(X_t)) g(x) m(dx).
	$$
\end{proof}

By Proposition \ref{duality}, all classical  processes with  t.i.p. considered
in \cite{Gallardo-Yor-2005} and \cite{Lawi-2008} are self-dual:
Bessel processes and their powers, Dunkl processes, Wishart processes, non-colliding particle systems (Dyson Brownian motion, non-colliding BESQ particles).
\begin{rem}
	{Let $n\ge 2$} and let $X$ be a transient regular process with t.i.p., with non-symmetric function $\phi$. By Theorem \ref{IPtip},
	$X$ has an IP, whereas  a DIP for $X$ is unknown.
	This observation, together with Remark \ref{stable} shows
	that in the theory of space inversions of stochastic processes, both IP and DIP must be considered.	
\end{rem}


\section{{Applications\\
		}}\label{appli}

\subsection{{ Free scaled power Bessel processes}}

  Let $R^{(\nu)}$ be a Bessel process with index $\nu>-1$ and dimension $ \delta=2(\nu+1)$.
  A time scaled power Bessel process is realized as $((R^{(\nu)}_{\sigma^2 t})^{\alpha}, t\geq 0)$, where $\sigma>0$ and $\alpha \neq 0$ are real numbers.
  Let $\underline{\nu}$ and $\underline{\sigma}$ be  vectors of real numbers such that  $\sigma_i>0$ and  $\nu_i>-1$ for all $i=1,2, \cdots, n$,  and let $R^{(\nu_1)}$, $R^{(\nu_2)}$, $\cdots$, $R^{(\nu_n)}$ be independent Bessel processes of index $\nu_1, \nu_2, \cdots, \nu_n$, respectively.
  We call the process $X$ defined, for a fixed $t\geq 0$, by
  \[X_t:=\left((R_{\sigma_1^2 t}^{(\nu_1)})^{\alpha},(R_{\sigma_2^2 t}^{(\nu_2)})^{\alpha}, \cdots (R_{\sigma_n^2 t}^{(\nu_n)})^{\alpha} \right)\]  a {\it free scaled power Bessel process with indices $\underline{\nu}$, scaling parameters $\underline{\sigma}$} \it{and power} $\alpha$}, for short FSPBES$(\underline{\nu}, \underline{\sigma}, \alpha)$.
If we denote by $q_{t}^{\nu}(x,y)$ the density of the semi-group of a BES$(\nu)$ with respect to the Lebesgue measure, found in \cite{ry}, then the densities of a FSPBES$(\underline{\nu}, \underline{\sigma},\alpha)$ are given by
\begin{eqnarray}\label{FSB}
p_t(x,y)&=& \prod_{i=1}^n (1/\alpha)y_i^{\frac{1}{\alpha}-1}q_{\sigma_i^2 t}^{\nu_i}(x_i^{1/\alpha}, y_i^{1/\alpha})\\ \nonumber
&=&\prod_{i=1}^n (1/\alpha)y_i^{\frac{1}{\alpha}-1} \frac{x_i^{1/\alpha}}{\sigma_i^2 t} \left( \frac{y_i}{x_i}\right)^{(\nu_i+1)/\alpha} I_{\nu_i}\left(\frac{(x_iy_i)^{1/\alpha}}{\sigma_i^2t}\right)e^{-\frac{x_i^{2/\alpha}+y_i^{2/\alpha}}{2\sigma_i^2 t}}.
\end{eqnarray}
From \eqref{FSB} we read that $p_t(x,y)$ takes the form (\ref{densities}) with
\begin{align}\label{characteristics-SFPBES}
\left\{
\begin{array}{l l}
\phi(x,y)=\prod_{i=1}^n\frac{I_{\nu_i}(\frac{(x_iy_i)^{1/\alpha}}{\sigma_i^2 })}{((x_iy_i)^{1/\alpha}/\sigma_i^2 )^{\nu_i}}, & \\
\rho(x)=\sum_{i=1}^nx_i^{2/\alpha}/\sigma_i^2, &\\
\theta(y)= \frac{1}{\alpha^n (\prod_{i=1}^n \sigma_i)^\alpha}\prod_{i=1}^n \left(\frac{y_i}{ |\sigma_i|^{\alpha}} \right)^{2(1+\nu_i)/\alpha-1}. &
\end{array} \right.
\end{align}
It follows that the degree of homogeneity of $\theta$ is $\beta=2(n+\sum_{i=1}^{n}\nu_i)/\alpha-n$. If $X$ is a
FSPBES$(\underline{\nu}, \underline{\sigma}, \alpha)$
  then clearly  $\rho^{{1/2}}(X)$ is a Bessel process of dimension {$n\overline{\delta}=2n(\overline{\nu}+1)$},
  where
  $\overline{\delta}=(\sum_1^n \delta_i)/n$
  and  $\overline{\nu}=(\sum_1^n \nu_i)/n$.
   Note that with this notation $\overline{\nu}=\frac{\alpha}{2n}(\beta+n)-1$
and $\overline{\delta}=\frac{\alpha}{n}(\beta+n)$ .
We deduce that $\rho(X)$ is point-recurrent if and only if $0<2n(\overline{\nu}+1)<2$, i.e., $0<n\overline{\delta}<2$.

{Interestingly, the distribution of $X_t$, for a fixed $t>0$, depends on the vector $\underline{\nu}$ only through the mean $\overline{\nu}$. Furthermore, we can recover the case $\sigma_1\neq 1$ from the case $\sigma_1=1$ by using the scaling property of Bessel processes. In other words, for a fixed time $t>0$, the class of all free power scaled Bessel processes yields an $(n+1)$-parameter family of distributions.}
{
	\begin{corollary}
		Let $X$ be a
		FSPBES$(\underline{\nu}, \underline{\sigma}, \alpha)$. If $n\overline{\delta}=2n(\overline{\nu}+1)>2$
		then $X$ is transient and has the Inversion Property with characteristics
			$$
			I(x)=\frac{x}{\rho^{\alpha}(x)},\quad
			h(x)=\rho^{1-\frac{n\overline{\delta}}{2}}(x),
			\quad
			v(x)=\rho(x)^2,
			$$
			{ where $\rho(x)$ is given by (\ref{characteristics-SFPBES}).}	
	\end{corollary}
 \begin{proof}
 We quote from (\cite{Lebedev}, p.136) that the modified Bessel function of the first kind $I_{\nu}$ has the asymptotics, for $\nu\geq 0$,
\[I_{\nu}(x)\sim \frac{x^{\nu}}{2^{\nu} \Gamma(1+\nu)} \quad \hbox{as} \quad x\rightarrow 0, \]
and
\[I_{\nu}(x)\sim \frac{e^x}{\sqrt{2\pi x}} \quad \hbox{as} \quad x\rightarrow \infty. \]
From the above and (\ref{FSB}) it follows that
\[p_t(x,y)\ \sim \ \frac{c(x,y)}{t^{n(1+\overline{\nu})}}\ \ \ \mbox{as}\ \ \ \ t\to  \infty,\]
and
\[p_t(x,y)\ \sim \ \frac{c(x,y)e^{- \frac{\rho(x)+\rho(y)}{2t}}}{t^{n/2}} \ \ \ \mbox{as}\ \ \ \ t\to  0,\]
hence if $n\overline{\delta}=2n(\overline{\nu}+1)>2$, then $\int_0^{\infty} p_t(x,y)\,dt<\infty$ and the process is transient.
The process   $\rho^{{1/2}}(X)$ is a Bessel process of dimension $2n(\overline{\nu}+1)=(\beta+n)\alpha$, so the condition
\eqref{condition-Bessel} is satisfied and
 we can apply Theorem \ref{IPtip}.

  We compute the Jacobian $\Jac(I)(x)=-\rho(x)^{-n\alpha}$ similarly as the Jacobian of the spherical inversion $x\mapsto x/\|x\|^2$ and we get
 $v(x)={|(\Jac(I)(x))^{-1}|\rho(x)^{2-n\alpha}}=\rho(x)^2$.
 \end{proof}
}
\subsection{{Gaussian Ensembles}} Stochastic Gaussian Orthogonal Ensemble GOE$(m)$ is an important class of processes with values in the space of real symmetric matrices $Sym(m,\R)$ which have t.i.p. and IP.
Recall that $$
Y_t=\frac{N_t+N_t^T}{2}, {t\geq 0,}
$$
where{$(N_t, t\geq 0)$} is a Brownian $m\times m$ matrix. Thus the upper triangular  processes  {$(Y_{ij})_{1\le i\le j\le m}$} of $Y$ are independent,
$Y_{ii}$ are Brownian motions and $Y_{ij}$, $i<j$, are  Brownian motions dilated by $\frac1{\sqrt{2}}$.

Let $M\in Sym(m,\R)$. We denote by ${\bf x}\in \R^m$ the diagonal elements of $M$
and by   ${\bf y}\in \R^{m(m-1)/2}$ the terms  $(M_{ij})_{1\le i<j\le m}$  above the diagonal of $M$.  We denote by  $M({\bf x,y})$ such a matrix $M$.

We have $({\bf x,y}) \in \R^{m(m+1)/2}$ and the map $({\bf x,y})\mapsto M({\bf x,y})$ is an isomorphism between  $\R^{m(m+1)/2}$ and  $Sym(m,\R)$.

Let $\Phi({\bf x,y})={M}({\bf x,y}/\sqrt{2})$. The map  $\Phi$ is a bijection of
$\R^{m(m+1)/2}$ and  $Sym(m,\R)$, such that the image of the Brownian Motion {$B$}
on $\R^{m(m+1)/2}$ is equal to {$Y$}. Proposition \ref{bijection} implies
 that  $Y$ has IP. More precisely, we obtain the following
\begin{corollary}\label{SGE}
The Stochastic Gaussian Orthogonal Ensemble GOE$(m)$ has IP with  characteristics:
$$
I(M)=\frac{M}{\|M\|^2},\quad h(M)=\|M\|^{2-n},\quad v(M)=\|M\|^4,
$$
where $\|M\|=\sqrt{\sum_{1\le i,j\le m} M_{ij}^2}$.
\end{corollary}

{On the other hand}, the time inversion property of $Y$ follows from the expression of the transition semigroup of $Y$ which is straightforward.
 Theorem \ref{IPtip} provides another proof of Corollary \ref{SGE}.\\

Analogously, IP and t.i.p. hold true for Gaussian Unitary and Symplectic Ensembles.

\subsection{ { Wishart Processes}}\label{Wishart}
 Now we look at matrix squared Bessel processes which are also known as Wishart  processes. Let $S^+_m$ be the set of $m\times m$ real non-negative definite matrices.  $X$ is said to be a   Wishart process with shape parameter $\delta$, if it satisfies the stochastic differential equation
$$dX_t=\sqrt{X_t}dB_t+dB^*_t\sqrt{X_t}+\delta I_m dt, \quad X_0=x,\quad
 \delta\in \{1,2,\ldots,m-2\}\cup [m-1,\infty), $$
where $B$ is an $m\times m$ Brownian matrix whose entries are independent linear Brownian motions, and $I_m$ is the $m\times m$ identity matrix. Notice that when $\delta$ is a positive integer, the Wishart process is the process $N^*N$ where $N$ is a $\delta\times m$ Brownian matrix process and $N^*$ is the transpose of $N$.
We refer to \cite{Donati-Doumerc-Matsumoto-2004} for Wishart processes.

 In \cite{Gallardo-Yor-2005} and \cite{Lawi-2008} it was shown that these processes have the t.i.p.
 The semi-group of $X$ is absolutely continuous with respect to the Lebesgue measure, i.e. $dy=\prod_{i\leq j}dy_{ij}$, with transition probability densities
\begin{equation*}
q_{\delta}(t,x,y)=\frac{1}{(2t)^{\delta m/2}}\frac{1}{\Gamma_m(\delta/2)}e^{-\frac{1}{2t} {\rm Tr} (x+y)}\left(\det(y) \right)^{(\delta-m-1)/2} {}_0F_1(\frac{\delta}{2}, \frac{xy}{4t^2}),
\end{equation*}
for $x, y\in S^+_m$, where $\Gamma_m$ is the multivariate gamma function and ${}_0F_1(\cdot, \cdot)$ is the matrix hypergeometric function.
 In particular, we have  $\rho(x)= {\rm Tr}(x)$, $\alpha=2$ ($X$ is self-similar with index 1)
 {and $\beta=\frac12 m(\delta-m-1)$}.
Observe that, by Proposition \ref{duality},  the Wishart process is self-dual with respect to the  Riesz measure
$$\theta(y) dy= \left(\det(y) \right)^{(\delta-m-1)/2} dy, \quad  y\in S_m^+,$$
 generating the Wishart family of laws of  $X$ as a natural exponential family.
Next, $X$ is transient for $m\ge 3$ and for $m=2$ and $\delta\ge 2$. For a proof of this fact, we use the s.d.e. of the trace of $X$
   given by
$$
d({\rm Tr}(X_t))=2 \sqrt{{\rm Tr}(X_t)} dW_t + m\delta dt.
$$
Thus, ${\rm Tr}(X)$ is a 1-dimensional  squared Bessel process of dimension $m\delta$.
{Since
	}$ \delta\in \{1,\ldots,m-2\}\cup [m-1,\infty)$, we have
$\delta \ge 1$, so $m\delta\ge 3$ unless, possibly the case $m=2$ and $\delta=1$. Thus, for $m\ge 3$ and for $m=2$ and $\delta\ge 2$, we have
$
\|X_t\|_1
=\sum_{i,j} |(X_t)_{ij}| \ge {\rm Tr}(X_t) \rightarrow \infty$ as $t \rightarrow \infty$ and the process $X$ is transient.

\begin{corollary}\label{CORWishart}	
Let $X$ be a Wishart process on  $S^+_m$, with shape parameter $\delta$.
The process $X$ has the  IP property with characteristics
$$
I(x)=\frac{x}{ ({\rm Tr}(x))^2},\quad
h(x)=({\rm Tr}(x))^{1-\frac{\delta m}{2}},\quad v(x)=\frac1{m-1}{\rm (Tr}(x))^2.$$
The function $h(x)=({\rm Tr}(x))^{1-\frac{\delta m}{2}}$ is $X$-excessive.
\end{corollary}
\begin{proof}
In the transient case we apply  Theorem \ref{IPtip}.
{Condition \eqref{condition-Bessel} is fulfilled as    $\rho(X)={\rm Tr}(X)$ is a 1-dimensional  squared Bessel process of dimension $m\delta{=(n+\beta)\alpha},$ where $n=m(m+1)/2$.}
 For the time change function,
the computation of the Jacobian of $I(X)$ is crucial. It is equal to $(m-1) ({\rm Tr}(X))^{-m(m+1)}$.

 In the case $m=2$ and $\delta=1$ it is easy to see that  the process  $X$ is not transient,
e.g. by checking that the integral {$\int_0^\infty q_{\delta}(t,0,y) dt=\infty$}.
Nevertheless, the IP holds with the same characteristics as above.  In order to prove this we can use
the following description of the generator of $X$ found in  \cite{Bru}. If $f$ and $F$ are $C^2 $ functions on, respectively, $\mathcal{S}_2^+$ and on $\mathcal{M}(1,2)$, the space of $1\times 2$ real matrices,  such that for all $y\in \mathcal{M}(1,2)$
we have $F(y)=f(y^*y),$ then $Lf=\frac12\Delta f$. Thus, the proof of the IP works like the one for the 2-dimensional Brownian motion, see \cite{yo}.
\end{proof}

\subsection{Dyson Brownian Motion}
Let $X_1\le X_2<\cdots \le X_n$
 be the ordered sequence of the eigenvalues of a Hermitian Brownian motion. Dyson showed in \cite{Dyson} that
 the process $(X_1,\ldots, X_n)$  has the same distribution as
  $n$ independent {{real-valued}}  Brownian motions conditioned never to collide. Hence its semigroup densities $p_t(x,y)$ can be described as follows. Let $q_t$ be the {{probability}} transition function of a  {{real-valued}} Brownian motion. We have
\begin{equation}\label{dyson-densities}
p_t(x,y)= \frac{H(y)}{H(x)}\det[q_t(x_i,y_j)], \quad x, y\in \mathbb{R}_{<}^n,
\end{equation}
where
$$ H(x)=\prod_{i<j}^{n} (x_j-x_i) \ \ \ \
{\rm and } \ \ \ \
\mathbb{R}^n_{<}=\{ x\in \mathbb{R}^n; x_1<x_2<\cdots <x_n\}. $$
Following Lawi \cite{Lawi-2008}, $X$ has the time inversion property. This follows from the fact that (\ref{dyson-densities}) can be written in the form (\ref{densities}) with
$$\theta=(2\pi)^{n/2}{H(y)^2}, \quad \rho(x)=\| x\|^2, \quad \phi(x,y)=\frac{\det [e^{x_i y_j}]_{i,j=1}^n}{{H(x)H(y)}}. $$
\begin{corollary}
The $n$-dimensional Dyson Brownian Motion has IP with characteristics:

 $I$ is the spherical inversion on $\mathbb{R}^n_{<}$, $h(x)=\|x\|^{2-{n^2}}$ and  {$v(x)=\|x\|^{4}$}.	
\end{corollary}
\begin{proof}
{We compute $(n+\beta)\alpha=n^2$}.
Applying Theorem \ref{IPtip} {to the Dyson Brownian Motion will be justified if we prove that $\|X\|^2$ is  BESQ($n^2$). This can be shown by writing the SDE for  $\|X\|^2$, using
	the SDEs for $X_i$'s and the It\^o formula.\\
	Another proof consists in observing that
	$H$ is harmonic for the $n$-dimensional Brownian Motion $B$ killed when it exits the set $\mathbb{R}_{<}^n$. It is also used in the construction of a Dyson Brownian Motion as a conditioned Brownian motion.
	An application of Proposition \ref{conditioning} yields the Corollary.}
\end{proof}
\subsection{{Non-colliding Squared Bessel Particles }}
Let $X_1\le X_2<\cdots \le X_n$
 be the ordered sequence of the eigenvalues of a  complex Wishart process, called a Laguerre process.
 K\"onig and O'Connell showed in \cite{KOC} that
 the process $(X_1,\ldots, X_n)$  has the same distribution as
  $n$ independent BESQ($\delta$) processes on $\R^+$ conditioned never to collide, $\delta>0$.
   Hence its semigroup {{densities}} $p_t(x,y)$ can be described as follows. Let $q_t$ be the {{probability}} transition function of a
  BESQ($\delta$) process. We have
\begin{equation}\label{laguerre-densities}
p_t(x,y)= \frac{H(y)}{H(x)}\det[q_t(x_i,y_j)], \quad x, y\in \mathbb{R^+}_{<}^{n},
\end{equation}
where  $H$ is, as in the previous example, the Vandermonde function
and $E=
\mathbb{R^+}^n_{<}=\{ x\in \mathbb{R^+}^n: x_1<x_2<\cdots <x_n\}. $
Lawi \cite{Lawi-2008} observed that $X$ has the time inversion property.

The same two reasonings  presented for the Dyson Brownian Motion can be applied,
in order to prove that  $X$ has IP. However, the first reasoning, using Theorem   \ref{IPtip}
and formula \eqref{laguerre-densities},
applies only in the transient case $\delta>2$.

 Let us present the second reasoning {where we use} the results of the Section  \ref{basic}. First, we prove the following corollary.
 \begin{corollary}
  The $n$-dimensional free Squared Bessel process $Y=(Y^{(1)},\ldots,Y^{(n)})$ where the processes $Y^{(i)}$ are independent  Squared Bessel processes of dimension $\delta$, has IP with characteristics
 $I(x)=x/(x_1+\ldots x_n)^2$,
   $h(x)=(\sum_{i=1}^n x_i)^{1-n\delta/2}$ and $v(x)=(\sum_{i=1}^n x_i)^2$.
 \end{corollary}
\begin{proof}
  It is an application of {the fact that a} free Bessel process has IP, {as} proved in \cite[Corollary 4]{acgz} , and   Proposition \ref{bijection}. We use the bijection
 $\Phi(x_1,\ldots,x_d)= (x_1^2,\ldots,x_d^2).$
\end{proof}
Next, {we apply  Proposition \ref{conditioning}, with $H$ as above, in order to get  the following result}.
 \begin{corollary}
 	Let $X=(X_1,\ldots, X_n)$ be
 	$n$ independent BESQ($\delta$) processes on $\R^+$ conditioned never to collide, $\delta>0$.
The process $X$ has IP with characteristics:\\
$$ I(x)=x/(x_1+\ldots + x_n)^2, \  \tilde h(x)=(\sum_{i=1}^n x_i)^{1-n\delta/2-n(n-1)},
\ \ v(x)=(\sum_{i=1}^n x_i)^2.$$
 \end{corollary}
 {
 \subsection{Dunkl processes}   \label{Dunkl}
 Let $R$ be a finite root system on $\mathbb{R}^n$. If $\alpha\in R$, then $\sigma_\alpha$ denotes  the symmetry with respect to the hyperplane $\{\alpha=0\}$.
 The Dunkl derivatives are defined by
 $T_if(x):= \partial_i f(x) +
 \sum_{ \alpha \in R^+} k(\alpha) \alpha_i \frac{f(x)-f(\sigma_\alpha x)}{\alpha\cdot x}$,    $i=1,2, \cdots, n$.
 The generator of a Dunkl process $X$ is $\frac12 \Delta_k$ where $\Delta_k=\sum_{i=1}^n T_i^2$ is the Dunkl Laplacian  on $\R^n$.
 \\
It was proven in \cite[Corollary 9]{acgz}  that any {Dunkl process $X_t$} has the {IP} with characteristics
	$I_{sph}$,  $h(x)= \|x\|^{2-n-2\gamma}$, where $\gamma=\frac12\sum_{\alpha\in R} k(\alpha)$, and
 $v(x)=\|x\|^4$.
 \\
 It is known(\cite{Gallardo-Yor-2005, AngersDunkl}) that  Dunkl processes are regular processes with t.i.p. Thus,  Theorem
 \ref{IPtip} provides an alternative method of proof of IP for transient Dunkl processes, characterized in \cite{GallardoTransient}.  By Theorem \ref{PROPKelvin}, we obtain the following corollary.
	\begin{corollary}
		Let $X$ be a Dunkl process on $\R^n$ and $h(x)= \|x\|^{2-n-2\gamma}$.
	The Kelvin transform ${\mathcal K}f=h\cdot f\circ I_{sph}$
	preserves $X$-harmonic, regular  $X$-harmonic and $X$-superharmonic functions.
	\end{corollary}
In \cite{CHM} the equivalence between operator-harmonicity $\Delta_ku=0$ and   	$X$-harmonicity of $u$ is  announced
and Kelvin transform for 	$X$-harmonic functions could be deduced
from \cite{KaYa}.
}

{ \subsection{ Hyperbolic Brownian Motion} \label{hyperb}
Let us recall {some} basic information about the ball realization of  real hyperbolic spaces (cf. \cite[Ch.I.4A p.152]{helg}, \cite{pyc}). The ball model of
 the real hyperbolic space of dimension $n$ is  the  $n$-dimensional Euclidean ball $\D^n=\{x\in \R^{n}:  \|x\|<1\}$
    equipped with the Riemannian metric
 $ds^2= 4\|dx\|^2/(1- \|x\|^2)^2$.
 The spherical coordinates on $\D^n$ are defined by $x= \sigma\tanh  \frac{r}{2}$ where $r>0$
 and $\sigma\in S^{n-1}\subset \R^n$ are unique. Then the Laplace-Beltrami operator on  $\D^n$
 is given by
 $$
L f(x)= \frac{\partial^2f}{\partial r^2}(x)+ (n-1) \coth r  \frac{\partial f}{\partial r}(x) + \frac{1}{\sinh^2r}\Delta_{S^{n-1}} f(x),
 $$
 where $\Delta_{S^{n-1}}$ is the spherical Laplacian on the sphere $S^{n-1}\subset \R^n$.\\
 Let $X$ be the $n$-dimensional Hyperbolic Brownian Motion on $\D^n$, defined as a diffusion generated by $\frac12 L$ (cf. \cite{pyc} and the references therein).
 {Define a new process $Y$ by setting} $Y_t:=\delta(X_t)$, $t\geq 0$, where $\delta(x)$ is the hyperbolic distance between $x\in \D^n$
 and the ball center ${\bf 0}$. The process $Y$ is the $n$-dimensional Hyperbolic Bessel process on $(0,\infty)$.
 According to \cite{agz}, the process $Y$ has the Inversion Property, with characteristics $(I_0, h_0, v_0)$ that can be determined  by  \cite[Theorem 1]{agz}. It is natural  to conjecture  that
 the  Hyperbolic Brownian Motion $X$ has IP with characteristics $(I, h, v_0)$, where
 $$I(x)=  \sigma\tanh{I_0(r)\over 2} \ \ {\rm and}\ \ h(x)=h_0(r).$$
When $n=3$, by   \cite[Section  5.2]{agz}, we have $I_0(r)=\frac12\ln\coth r$, $h_0(t)=\coth r -1$
and $v_0(r)= 2\cosh r\sinh r$.
If the  Hyperbolic Brownian Motion $X_t$ {had} IP with the {involution} $I$ and the excessive function $h$, then, by Theorem \ref{PROPKelvin} and  Proposition \ref{equiv_harmonic},
 if $L f=0$ then $L(h f\circ I)=0$. By a direct {but}  tedious calculation of $L(h f\circ I)$ in spherical coordinates, we see that there exist continuous functions $f$ such that
 $Lf=0$ but $L(h f\circ I)\not=0$, so $X$ does not have IP with characteristics $I$ and  $h$.

 To our knowledge, no inversion property {and Kelvin transform  are} known for
 the  Hyperbolic Brownian Motion. We believe that this question was first  raised by T. Byczkowski {about} ten years ago, while {he was} working on potential theory of the  Hyperbolic Brownian Motion (\cite{BGS}).


\section{Acknowledgements}
We are grateful to
K. Bogdan,  M. Kwa\'snicki and Z. Palmowski for discussions on potentials and different
definitions of harmonic functions and generators.
We thank the organizers of the Workshop "Stable processes" Oaxaca 2016
during which this paper was finalized and presented and we thank
 M.E. Caballero for stimulating discussions during the Workshop.
 We thank  two referees  for  their precious comments and bibliographical hints that improved and enriched our paper.

\end{document}